\documentclass[12pt]{amsart}

\usepackage{hyperref}
\usepackage{amssymb,amsthm,amsmath,amsrefs}

\usepackage[all,cmtip]{xy}
\usepackage{graphicx}
\usepackage{dsfont}

\usepackage{fullpage}
\newtheorem{theorem}{Theorem}
\newtheorem{corollary}{Corollary}
\newtheorem{lemma}{Lemma}
\newtheorem{proposition}{Proposition}

\theoremstyle{remark}
\newtheorem{example}[theorem]{Example}

\newcommand{\N}{\mathbb{N}}

\newcommand{\R}{\mathbb{R}}
\newcommand{\C}{\mathbb{C}}
\newcommand{\id}{\mathrm{id}}
\newcommand{\wayb}{\ll}

\newcommand{\cuntzle}{\preccurlyeq}
\newcommand{\Cu}{\mathrm{Cu}}
\newcommand{\CCu}{\mathbf{Cu}}

\DeclareMathOperator{\rank}{rank}

\renewcommand{\epsilon}{\varepsilon}
\renewcommand{\leq}{\leqslant}
\renewcommand{\geq}{\geqslant}

\title[On the classification of homomorphisms...]
{Classification of homomorphisms from $C_0(0,1]$  to a C*-algebra}

\author{Leonel Robert \and Luis Santiago}

\begin{document}
\begin{abstract}
A class of C*-algebras is described  for which 
the homomorphism from $C_0(0,1]$ to the algebra may be classified
by means of the Cuntz semigroup functor. Examples are given of 
algebras---simple and non-simple---for which this classification fails. It is shown that
a suitable suspension of the Cuntz semigroup functor deals successfully  with
some of these counterexamples. 
\end{abstract}

\maketitle

\section{Introduction}
In this paper we consider the question  of classifying the homomorphisms from $C_0(0,1]$ to 
a C*-algebra $A$. In \cite{ciuperca-elliott}, Ciuperca and Elliott show that 
if $A$ has stable rank 1 then this classification is possible---up to approximate unitary equivalence---by 
means of the the Cuntz semigroup functor. They define a pseudometric $d_W$ on the morphisms from $\Cu(C_0(0,1])$
to $\Cu(A)$, and show if $A$ has stable rank 1 then $d_W(\Cu(\phi),\Cu(\psi))=0$ for $\phi,\psi\colon C_0(0,1]\to A$  
if and only if  $\phi$ and $\psi$ are approximately unitarily equivalent by unitaries in $A^\sim$ (the unitization of $A$).

A classification result in the same spirit as Ciuperca and Elliott's result is Thomsen's \cite[Theorem 1.2]{thomsen}. Thomsen shows that if $X$ is a locally compact Hausdorff space such that $\dim X\leq 2$ and $\check H^2(X)=0$, then the 
approximate unitary equivalence class of  a positive element in $M_n(C_0(X))$ is determined by its eigenvalue
functions. 

Theorem \ref{1} below applies to a class of C*-algebras that contains both the stable rank 1 C*-algebras and 
the C*-algebras considered by Thomsen. For this class of algebras the classification of  homomorphisms 
by the functor $\Cu(\cdot)$  must be rephrased in terms of stable approximate unitary equivalence. Given  $\phi,\psi\colon C_0(0,1]\to A$ we say that $\phi$ and $\psi$ are stably approximately unitarily equivalent if there are unitaries $u_n\in (A\otimes \mathcal K)^\sim$, $n=1,2\dots$, such that $u_n\phi u_n^*\to \psi$ pointwise (where $A$ is identified with the top left corner of $A\otimes \mathcal K$). If $A$ is stable or has stable rank 1, then stable approximate unitary equivalence coincides with approximate unitary equivalence, but these relations might differ in general.

The following theorem  characterizes the C*-algebras for which the pseudometric $d_W$ (defined
in the next section) determines the stable approximate unitary equivalence classes of homomorphism 
from $C_0(0,1]$ to the algebra. 
\begin{theorem}\label{1}
 Let $A$ be a C*-algebra. The following propositions are equivalent.

(I) For all $x,e\in A$ with $e$ a positive contraction and $ex=xe=x$, we have that
$x^*x+e$ is stably approximately unitarily equivalent to $xx^*+e$.

(II) If $\phi,\psi\colon C_0(0,1]\to A$ are such that $d_W(\Cu(\phi),\Cu(\psi))=0$
then $\phi$ is stably approximately unitarily equivalent to $\psi$.

 If (I) and (II) hold then 
\begin{equation}\label{inequalities}
d_W(\phi,\psi)\leq d_{U}(\phi,\psi)\leq 4d_W(\phi,\psi).
\end{equation}
\end{theorem}
In \eqref{inequalities} $d_{U}$ denotes the distance between the stable unitary orbits of
$\phi(\mathrm{id})$ and $\psi(\mathrm{id})$, where $\mathrm{id}\in C_0(0,1]$ is the 
identity function. The inequalities \eqref{inequalities} are derived in 
\cite{ciuperca-elliott} for the stable rank 1 case, though their factor of 8 has now
been improved to 4. 

By the bijective correspondence $\phi\mapsto \phi(\mathrm{id})$ between homomorphisms
$\phi\colon C_0(0,1]\to A$ and positive contractions of $A$, the proposition (II) of the previous theorem may be
restated as a classification of the stable unitary orbits of positive contractions in terms of the Cuntz
equivalence relation of positive elements. 


The following theorem extends Ciuperca and Elliott's classification result beyond the stable rank 1 case.
\begin{theorem}\label{2}
Suppose that $(A\otimes \mathcal K)^{\sim}$ has the property  (I) of Theorem \ref{1}. Let $h_A\in A^+$ be strictly positive.
Then for every $\alpha\colon \Cu(C_0(0,1])\to \Cu(A)$,  morphism  in the category $\CCu$,
with $\alpha([\mathrm{id}])\leq [h_A]$, there is $\phi\colon C_0(0,1]\to A$, unique up to 
stable approximate unitary equivalence, such that $\Cu(\phi)=\alpha$. 
\end{theorem}

The class of algebras that satisfy (I) is closed under the passage to quotients, hereditary subalgebras,
and inductive limits (see Proposition \ref{classI}  below). This class is strictly larger than the class
of stable rank 1 C*-algebras. Any commutative C*-algebra satisfies (I). If $X$ is a locally compact Hausdorff space
with  $\dim X\leq 2$ and $\check H^2(X)=0$ (the Cech cohomology with integer coefficients), then we deduce from \cite[Theorem 1.2]{thomsen} that
$(C_0(X)\otimes \mathcal K)^\sim$ satisfies (I) (and so, Theorem \ref{2} is applicable to $C_0(X)\otimes \mathcal K$). 
On the other hand, the C*-algebra $M_2(C(S^2))$, with $S^2$ the 2-dimensional sphere, does not satisfy (I). In fact, there exists a pair of homomorphisms $\phi,\psi\colon C_0(0,1]\to M_2(C(S^2))$ such that $\Cu(\phi)=\Cu(\psi)$
but $\phi$ is not stably approximately unitarily  equivalent to $\psi$ (see Example \ref{sphere} below).
This phenomenon is not restricted to non-simple AH C*-algebras: by a slight 
variation---to suit our purposes---of the inductive limit systems constructed by Villadsen in \cite{villadsen}, we 
construct a simple, stable, AH C*-algebra for which the Cuntz semigroup functor does not classify the homomorphism from $C_0(0,1]$ into the algebra (see Theorem \ref{example}). 
These counterexamples raise the question of what additional data is necessary to classify,
up to stable approximate unitary equivalence, the homomorphisms from $C_0(0,1]$ to an arbitrary C*-algebra. 
In the last section of this paper we take a  step in this direction by proving the following theorem.

\begin{theorem}\label{extension}
Let $A$ be  an inductive limit of the form $\varinjlim C(X_i)\otimes\mathcal K$,
with $X_i$ compact metric spaces, and $\dim X_i\leq 2$ for all $i=1,2,\dots$. Let $\phi,\psi\colon C_0(0,1]\to A$ be homomorphisms such that $\Cu(\phi\otimes \mathrm{Id})=\Cu(\psi\otimes \mathrm{Id})$, 
where $\mathrm{Id}\colon C_0(0,1]\to C_0(0,1]$ is the identity homomorphism. Then $\phi$ and $\psi$ are 
approximately unitarily equivalent.
\end{theorem}

\section{Preliminary definitions and results}
In this section we collect a number of definitions and results that will be used
throughout the paper.
\subsection{Relations on positive elements.}
Let $A$ be a C*-algebra and let $a$ and $b$ be positive elements of $A$. 
Let us say that

(i) $a$ is Murray-von Neumann equivalent  to $b$ if there is $x\in A$ such that $a=x^*x$
and $b=xx^*$; we denote this by $a\sim b$,

(ii) $a$ is approximately Murray-von Neumann equivalent to $b$ if there are $x_n\in A$, $n=1,2\dots$, such that $x_n^*x_n\to a$
and $x_nx_n^*\to b$; we denote this by $a\sim_{ap} b$,

(iii) $a$ is stably approximately unitarily equivalent to $b$ if there are unitaries  $u_n\in (A\otimes \mathcal K)^\sim$, 
such that $u_n^*au_n\to b$, where $A$ is identified with the top left corner of 
$A\otimes \mathcal K$,

(iv) $a$ is Cuntz smaller than $b$ if there are $d_n\in A$, $n=1,2\dots$, such that $d_n^*bd_n\to a$;
we denote this by $a\cuntzle_{Cu} b$,  

(v) $a$ is Cuntz equivalent to $b$ if $a\cuntzle_{Cu} b$ and $b\cuntzle_{Cu} a$, and we denote this
by $a\sim_{Cu} b$.  

We have (i)$\Rightarrow$(ii)$\Rightarrow$(v). By \cite[Remark 1.8]{thomsen2}, approximate Murray-von Neumann
equivalence is the same as stable approximate unitary equivalence. We will make frequent use of this fact
throughout the paper. The relations (i),(ii), and (iii) will also be applied to homomorphisms from $C_0(0,1]$ to $A$, via the bijection $\phi\mapsto \phi(\mathrm{id})$ from these homomorphisms into the positive contractions of $A$.

We will make frequent use of the following proposition.

\begin{proposition}\label{rordam1} Let $a\in A^+$ and $x\in A$ be such that
$\|a-x^*x\|<\epsilon$ for some $\epsilon>0$. Then there is $y$ such that
$(a-\epsilon)_+=y^*y$, $yy^*\leq xx^*$, and $\|y-x\|<C\epsilon^{1/2}\|a\|$.
The constant $C$ is universal.
\end{proposition}
\begin{proof}
 The proof works along the same lines as the proof of \cite[Lemma 2.2]{kirchberg-rordam} (see also
\cite[Lemma 1]{robert}). We briefly skecth the argument here.
We have $a-\epsilon_1\leq x^*x$, with $\epsilon_1$ such that $\|a-x^*x\|<\epsilon_1<\epsilon$. 
So $(a-\epsilon)_+\leq ex^*xe$, with
$e\in C^*(a)$ such that $e(a-\epsilon_1)e=(a-\epsilon)_+$. Set $xe=\widetilde x$ and let $\widetilde x=v|\widetilde x|$
be its polar decomposition. Then $y=v(a-\epsilon)_+^{1/2}$ has the properties stated in the proposition.
\end{proof}

It follows from the previous proposition (or from \cite[Lemma 2.2]{kirchberg-rordam}), that Cuntz comparison can be described
in terms of Murray-von Neumann equivalence as follows: $a\cuntzle_{Cu} b$ if and only if for every $\epsilon>0$
there is $b'$ such that $(a-\epsilon)_+\sim b'\in \mathrm{Her}(b)$. Here $\mathrm{Her}(b)$ denotes the hereditary subalgebra
generated by $b$. We also have the following corollary
of Proposition \ref{rordam1}.

\begin{corollary}\label{mvnher}
If $a,b\in B^+\subseteq A^+$, where  $B$ is a hereditary subalgebra of $A$, then $a\sim_{ap} b$
in $A$ if and only if $a\sim_{ap} b$ in $B$. 
\end{corollary}
\begin{proof}
If $w^*w$ and $ww^*$ belong to $B$ for some $w\in A$, then $w\in B$.
Thus, the proposition follows if $a$ and $b$ are  Murray-von Neumann equivalent.

Suppose that $a\sim_{ap} b$.
We may assume without loss of generality that $a$ and $b$ are contractions.
For $\epsilon>0$ let $x\in A$ be such that $\|a-x^*x\|<\epsilon$ and $\|b-xx^*\|<\epsilon$. Then
by Proposition \ref{rordam1} there exists  $y$ such that $(a-\epsilon)_+=y^*y$ and $\|yy^*-b\|\leq C_1\sqrt{\epsilon}$
for some constant $C_1$. Applying Proposition \ref{rordam1} again we get that there exists $z\in A$ such that 
$(yy^*-\epsilon)_+=z^*z$, $\|zz^*-b\|\leq C_2\sqrt[4]{\epsilon}$, and $zz^*\leq b$, for some constant $C_2$. 
Set $zz^*=b'$. We have $(a-2\epsilon)_+\sim (yy^*-\epsilon)_+\sim b'$ and $b'\in B$. So there is $w\in B$ such that 
$(a-2\epsilon)_+=w^*w$ and $b'=ww^*$. Since  $\|b'-b\|\leq C_2\sqrt[4]{\epsilon}$ and $\epsilon$ is arbitrary, 
the desired result follows. 
\end{proof}

\subsection{The Cuntz semigroup.} Let us briefly recall the definition of the (stabilized) Cuntz semigroup 
in terms of the positive elements of the stabilization of the algebra (see \cite{rordam1} and \cite{coward-elliott-ivanescu}). Let $A$ be a C*-algebra. Given 
$a\in (A\otimes \mathcal K)^+$ let us denote by $[a]$ the Cuntz equivalence class of $a$. The Cuntz semigroup of $A$ is defined as the set of Cuntz equivalence classes of positive elements of $A\otimes \mathcal K$. This set, denoted by $\Cu(A)$, is endowed with the order such that $[a]\leq [b]$ if $a\cuntzle_{Cu} b$, and the addition operation $[a]+[b]:=[a'+b']$, where $a'$ and $b'$ are mutually orthogonal and Murray-von Neumann equivalent to $a$ and $b$, respectively.

If $\phi\colon A\to B$ then $\Cu(\phi)\colon \Cu(A)\to \Cu(B)$ is defined by $\Cu(\phi)([a]):=[\phi(a)]$.
Coward, Elliott, and Ivanescu, showed in \cite{coward-elliott-ivanescu} that $\Cu(\cdot)$ is a functor
from the category of C*-algebras to a certain category of ordered semigroups denoted by $\CCu$. In order to describe this category, let us first recall the definition
of the far below relation. Let $S$ be an ordered set such that the suprema of increasing sequences always exists in $S$. For $x$ and $y$ in $S$, let us 
say that $x$ is far below $y$, and denote it by $x\ll y$, if for every increasing sequence $(y_n)$ such that $y\leq \sup_n y_n$, we have 
$x\leq y_k$ for some $k$. 

An  ordered semigroups $S$ is an object of the Cuntz category $\CCu$ if it has a 0 element and satisfies that

(1) if $(x_n)$ is an increasing sequence of elements of $S$ then $\sup_n x_n$ exists in $S$,

(2) if $(x_n)$ and $(y_n)$ are increasing sequences in $S$ then $\sup_n (x_n+y_n)=\sup_n x_n+\sup_n y_n$,

(3) for every $x\in S$ there is a sequence $(x_n)$ with supremum $x$ and such that $x_n\ll x_{n+1}$
for all $n$,

(4) if $x_1,x_2,y_1,y_2\in S$ satisfy $x_1\ll y_1$ and $x_2\ll y_2$, then $x_1+x_2\ll y_1+y_2$.

\noindent 
The morphisms of the category $\CCu$ are the order preserving semigroup maps that also preserve the suprema of increasing sequences, 
the far below relation, and the 0 element. 

\subsection{The pseudometrics $d_U$ and $d_W$.}
Let us identify the C*-algebra $A$ with the top left corner of $A\otimes \mathcal K$.
Given positive elements $a,b\in  A$ let us denote by $d_U(a,b)$ the distance between the unitary
orbits of $a$ and $b$ in $A\otimes \mathcal K$ (with the unitaries taken in $(A\otimes \mathcal K)^\sim$).

Following Ciuperca and Elliott (see  \cite{ciuperca-elliott}), let us define a pseudometric on the morphisms from $\Cu(C_0(0,1])$ to $\Cu(A)$ as follows:  
\begin{align}\label{defdWmor}
d_W(\alpha,\beta):=\inf\left\{r\in \R^+\left| 
\begin{array}{c}
\alpha([e_{t+r}])\le\beta([e_t]),\\
\beta([e_{t+r}])\le\alpha([e_t]),
\end{array}
\hbox{ for all }t\in \R^+\right\},
\right.
\end{align}
where $\alpha,\beta\colon \Cu(C_0(0,1])\to \Cu(A)$ are morphisms in the Cuntz category
and $e_t$ is the function $e_t(x)=\max(x-t,0)$, for $x\geq 0$. It is easily shown that $d_W$ is a pseudometric.

\emph{Notation convention.} All throughout the paper
we will use the notations $(a-t)_+$ and $e_t(a)$ interchangeably, both meaning the positive
element obtained evaluating the function $e_t(x)$ on a given selfadjoint element $a$.

The pseudometric $d_W$ may be used to define a pseudometric---that we also denote by $d_W$---on 
the positive elements of norm at most 1  by setting $d_W(a,b):=d_W(\Cu(\phi),\Cu(\psi))$, where $\phi,\psi\colon C_0(0,1]\to A$ are
such that $\phi(\id)=a$ and $\psi(\id)=b$. We have
\begin{align}\label{defdW}
d_W(a,b)=\inf\left\{r\in \R^+\left| 
\begin{array}{c}
e_{t+r}(a)\cuntzle_{Cu} e_t(b),\\
e_{t+r}(b)\cuntzle_{Cu} e_t(a),
\end{array}
\hbox{ for all }t\in \R^+\right\}.
\right.
\end{align}
Notice that \eqref{defdW} makes sense for arbitrary positive elements $a$ and $b$ without assuming that they are contractions. 
We extend $d_W$ to all positive elements using \eqref{defdW}.

The following lemma relates the metrics $d_U$ and $d_W$ in a general C*-algebra 
(this is \cite[Corollary 9.1]{ciuperca-elliott}).

\begin{lemma}\label{continuous}
For all $a,b\in A^+$ we have $d_{W}(a,b)\leq d_U(a,b)\leq \|a-b\|$.
\end{lemma}
\begin{proof}
Let $r$ be such that $\|a-b\|<r$ and choose $r_1$ such that $\|a-b\|<r_1<r$. 
Then for all $t\geq 0$ we have $a-t-r_1\leq b-t$. Multipliying this inequality on the left
and the right by $e^{1/2}$, where $e\in C^*(a)$ is such $e(a-t-r_1)=(a-t-r)_+=e_{t+r}(a)$, we have  
\[
e_{t+r}(a)\leq e^{1/2}(b-t)e^{1/2}\leq e^{1/2}(b-t)_+e^{1/2}\cuntzle_{Cu} e_t(b),
\] 
for all $t\geq 0$. Similarly we deduce that  $e_{t+r}(b)\cuntzle_{Cu} e_t(a)$ for all $t\geq 0$. It follows that $d_W(a,b)\leq \|a-b\|$. Since $d_{W}$ is
invariant by stable unitary equivalence, $d_W(a,b)\leq \|a-ubu^*\|$ for any $u$ unitary in $(A\otimes \mathcal K)^\sim$. Hence $d_W(a,b)\leq d_U(a,b)$. 
\end{proof}

The question of whether $d_W$---as defined in \eqref{defdWmor}---is a metric 
is linked to the property of weak cancellation in the Cuntz semigroup. Let us say that a semigroup in the category $\CCu$ has weak cancellation if $x+z\wayb y+z$
implies $x\leq y$ for elements $x$, $y$, and $z$ in the semigroup. It was proven in \cite{ciuperca-elliott} that
if $\Cu(A)$ has weak cancellation then $d_W$ is a metric on the morphisms from $\Cu(C_0(0,1])$ to $\Cu(A)$. 
Since this result is not explicitly stated in that paper, we reprove it here.

\begin{proposition} (Ciuperca, Elliott \cite{ciuperca-elliott})
 If $Cu(A)$ has weak cancellation then $d_W$ is a metric on the Cuntz category morphisms from $\Cu(C_0(0,1])$ to $\Cu(A)$.
\label{metric}
\end{proposition}
\begin{proof}
 By \cite[Theorem 1]{robert}, the map $[f]\mapsto (t\mapsto \rank f(t))$ is a well defined
isomorphism from $\Cu(C_0(0,1])$ to the ordered
semigroup of lower semicontinuous functions from $(0,1]$ to $\N\cup\{\infty\}$. This isomorphism
maps $[e_t]$ to $\mathds{1}_{(t,1]}$ for all $t\in [0,1]$, with
$\mathds{1}_{(t,1]}$ the characteristic function of $(t,1]$. Let us identify 
$\mathrm{Cu}(C_0(0,1])$ with the
semigroup of lower semicontinuous functions from $(0,1]$ to $\N\cup\{\infty\}$ in this way. Then $d_W(\alpha,\beta)=0$
says that $\alpha(\mathds{1}_{(t,1]})=\beta(\mathds{1}_{(t,1]})$ for all $t$. 
In order to show that $\alpha$ and $\beta$ are equal it suffices  to show that they agree
on the functions $\mathds{1}_{(s,t)}$ (their overall equality then follows by additivity and preservation of suprema of 
increasing sequences). 

Let $\epsilon>0$. We have
\begin{align*}
\alpha(\mathds{1}_{(s+\epsilon, t-\epsilon)})+\alpha(\mathds{1}_{(t-\epsilon, 1]})&\wayb \alpha(\mathds{1}_{(s, 1]})=\beta(\mathds{1}_{(s, 1]})\\&\le \beta(\mathds{1}_{(s,t)})+\beta(\mathds{1}_{(t-\epsilon,1]})\\
&=\beta(\mathds{1}_{(s,t)})+\alpha(\mathds{1}_{(t-\epsilon,1]}).
\end{align*}
Since $A$ has weak cancellation $\alpha(\mathds{1}_{(s+\epsilon, t-\epsilon)})\le\beta(\mathds{1}_{(s,t)})$.
Passing to the supremum over $\epsilon>0$ we get that  $\alpha(\mathds{1}_{(s, t)})\le\beta(\mathds{1}_{(s,t)})$. 
By symmetry we also have  $\beta(\mathds{1}_{(s,t)})\le\alpha(\mathds{1}_{(s, t)})$. Hence, 
$\alpha(\mathds{1}_{(s, t)})=\beta(\mathds{1}_{(s,t)})$.
\end{proof}

R\o rdam and Winter showed in \cite[Theorem 4.3]{rordam-winter} that if $A$ has stable rank 1 then $\Cu(A)$ has weak cancellation. In the next section we will extend this result to the case when the property (I) of Theorem \ref{1}
 holds in $(A\otimes\mathcal K)^\sim$.

\section{Proofs of Theorems 1 and 2}

\subsection{Proof of Theorem 1} In this subsection we prove Theorem \ref{1} of the introduction.

For positive elements $a,b\in A^+$ we use the notation $a\lhd b$ to mean that $b$ is a unit for $a$, that is
to say, $ab=ba=a$. We start with a lemma.

\begin{lemma}
Let $A$ be a C*-algebra such that the property (I) of Theorem \ref{1} holds in $A$.
Let $e,f,\alpha,\beta\in A^+$ be such that $e$ is a contraction, and 
\[
\alpha\lhd e,\, \alpha\sim \beta\lhd f, \hbox{ and $f\sim f'\lhd e$ for some $f'\in A^+$.}
\]
Then for every $\delta>0$ there are $\alpha',e'\in A^+$ such that
\[\alpha'\lhd e'\lhd e, \quad\beta+f\sim \alpha'+e', \hbox{ and }\|\alpha-\alpha'\|<\delta.\] 
\end{lemma}
\begin{proof}
Since $f\sim f'$ there exists $x$ such that $f=x^*x$ and $xx^*=f'$. Let $x=w|x|$ be the polar decomposition of $x$
in the bidual of $A$. We have $wfw^*=f'$. Set $w\beta w^*=\alpha_1$. Then 
$\alpha_1\sim \alpha$, $\alpha_1\lhd e$, and $\alpha\lhd e$. Hence $\alpha_1+e\sim_{ap} \alpha+e$. 
By Proposition \ref{rordam1} this implies that for every
$\delta'>0$ there is $z\in A$ such that 
\begin{align}
&(\alpha_1+e-\delta')_+=z^*z, \quad zz^*\leq \alpha+e,\hbox{ and}\\
&\| zz^*-(\alpha+e)\|<C\sqrt{\delta'}. \label{z}
\end{align}
Let $z=w_1|z|$ be the polar decomposition of $z$ in the bidual of $A$.
Since $e$ is a unit for $\alpha_1$ we have $(\alpha_1+e-\delta')_+=\alpha_1+(e-\delta')_+$. It follows that the map
$c\mapsto w_1cw_1^*$, sends the elements of $\mathrm{Her}((e-\delta')_+)$ into $\mathrm{Her}(e)$.
By \eqref{z} if we let $\delta'\to 0$ then $(zz^*-1)_+$ can be made arbitarily close to  $(\alpha+e-1)_+$.
Since $(zz^*-1)_+=w_1(a_1-\delta')_+w_1^*$ and $(\alpha+e-1)_+=\alpha$, this means that we can choose $\delta'$
small enough so that $\|w_1\alpha_1w_1^*-\alpha\|<\delta$. Let $\alpha'=w_1\alpha_1w_1^*$, $e'=w_1f'w_1^*$, and $y=w_1w(\beta+f)^{1/2}$. Then $\beta +f=y^*y$ and $yy^*= \alpha'+e'$. 
\end{proof}

\begin{proof}[Proof of Theorem \ref{1}]
(II) $\Rightarrow$ (I). Let $\phi, \psi\colon C_0(0,1]\to A$ be the homomorphism such that 
\[
\phi(\mathrm{id})=\frac{1}{\|x\|^2+1}(x^*x+e)\hbox{ and } \psi(\mathrm{id})=\frac{1}{\|x\|^2+1}(xx^*+e).
\]
From the definition of the pseudometric $d_W$ we see $d_W(\Cu(\phi),\Cu(\psi))=\frac{1}{|x|^2+1}d_W(x^*x+e, xx^*+e)$. In order to prove that $x^*x+e$ is stably approximately unitarily equivalent to $xx^*+e$ it is enough to show that 
\[
d_W(x^*x+e, xx^*+e)=0.
\]
That is, $(x^*x+e-t)_+\sim_{Cu} (xx^*+e-t)_+$ for all $t\in \R$.

Using that  $e$ is a unit for $x^*x$ and $xx^*$ we deduce that 
\begin{align*}
(x^*x+e-t)_+=x^*x+(e-t)_+, \quad (xx^*+e-t)_+=xx^*+(e-t)_+,
\end{align*}
for $0\le t<1$.  Also, $x^*x(e-t)_+=x^*x(1-t)$ and $xx^*(e-t)_+=xx^*(1-t)$. It follows that $x^*x$ and $xx^*$ belong to the hereditary algebra generated by $(e-t)_+$. Therefore, 
\[
(x^*x+e-t)_+\sim_{Cu}(e-t)_+\sim_{Cu} (xx^*+e-t)_+, \hbox{ for }0\leq t<1.
\]
If $t\ge1$ then $(x^*x+e-t)_+=(x^*x+1-t)_+$ and $(xx^*+e-t)_+=(xx^*+1-t)_+$. Hence, $(x^*x+e-t)_+\sim_{Cu} (xx^*+e-t)_+$
for $t\geq 1$.

(I) $\Rightarrow$ (II).
Set $\phi(\mathrm{id})=a$ and $\psi(\mathrm{id})=b$. Let $r$ be such that $d_W(a,b)<r$.
Let $m\in \N$ be the number such that $mr\leq 1<(m+1)r$. Finally, let the sequences $(a_i)_{i=1}^{m+1}$, $(b_i)_{i=1}^{m+1}$
be defined as $a_i=\xi_{m-i+1}(a)$, $b_i=\xi_{m-i+1}(b)$ for $i=1,2,\dots,m+1$, where $\xi_i\in C_0(0,1]$
is such that $\mathds{1}_{(ir+\epsilon,1]}\le\xi_i\le\mathds{1}_{(ir,1]}$ and $\epsilon>0$ is chosen small enough so that 
$d_W(a,b)+2\epsilon<r$. 

The sequences $(a_i)_{i=1}^{m+1}$ and $(b_i)_{i=1}^{m+1}$ satisfy that 
\begin{align*}
 a_i\lhd a_{i+1}, \quad b_i\lhd b_{i+1}, \hbox{ for $i=1,\dots,m$},\\
a_i\sim d_i\lhd b_{i+1},\quad b_i\sim c_i\lhd a_{i+1},\hbox{ for $i=1,\dots,m$},
\end{align*}
for some positive elements $c_i$ and $d_i$. The first line follows trivially from the definition
of the elements $a_i$ and $b_i$. Let us prove the second line. From
$d_W(a,b)<r-2\epsilon$ we get
\[
e_{(m-i+1)r-\epsilon}(a)\cuntzle_{Cu} e_{(m-i)r+\epsilon}(b)\lhd b_{i+1}.
\]
By the definition of Cuntz comparison there exists $d\in A^+$ such that  
$e_{(m-i+1)r}(a)\sim d\lhd b_{i+1}$. Since $a_i$ is expressible by functional
calculus as a function of $e_{(m-i+1)r}(a)$, we get that there exists $d_i\in A^+$ such that 
$a_i\sim d_i\lhd b_{i+1}$. We reason similarly to get the existence of $c_i$.

Let us now show by induction on $n$, for $n=1,2,\dots,m$,  that there are sequences of elements
$(a_i')_{i=1}^n$ and $(b_i')_{i=1}^n$ such that
\begin{align}
 a_i'\lhd a_{i+1}', \quad b_i'\lhd b_{i+1}', &\hbox{ for }i=1,2,\dots n-1\label{kk1}\\
\|a_i-a_i'\|<\epsilon, &\hbox{ for $i$ odd, }i\leq n,\label{kk2}\\
\|b_i-b_i'\|<\epsilon, &\hbox{ for $i$ even, }i\leq n,\label{kk3}\\
\sum_{i=1}^n a_i' \sim &\sum_{i=1}^n b_i',\label{kk4}
\end{align}
and $a_n'=a_n$, $b_n'\lhd b_{n+1}$ if $n$ is odd, and $b_n'=b_n$, $a_n'\lhd a_{n+1}$ if $n$
is even.

Since $a_1\sim d_1\lhd b_2$, the induction hypothesis holds for $n=1$ taking $b_1'=d_1$.
Suppose the induction holds for $n$ and let us show that it also holds for $n+1$. 
Let us consider the case that $n$ is odd (the case that $n$ is even is dealt with similarly).
We set $b_{n+1}'=b_{n+1}$ and leave the sequence $(b_i')_{i=1}^n$ unchanged. We are
going to modify the squence $(a_i')_{i=1}^n$ in order to complete the induction step.
Let $\alpha=\sum_{i=1}^n a_i'$, $e=a_{n+2}$, $\beta=\sum_{i=1}^n b_i'$,
$f=b_{n+1}'$. Then the conditions of the previous lemma apply. We thus have that for every $\delta>0$
there are $\alpha'$ and $e'$, such that 
\[\alpha'\lhd e'\lhd a_{n+2}, \quad \|\alpha-\alpha'\|<\delta, \hbox{ and }\beta+f\sim \alpha'+e'. 
\]
It follows that $\beta\sim \alpha'$, and so $\alpha'=\sum_{i=1}^n a_i''$,
with $a_i''\lhd a_{i+1}''$. We remark that the elements  $a_i'$ are all in the C*-algebra generated
by $\alpha$ and the elements $a_i''$ are in the C*-algebra generated by $\alpha'$. In fact, 
\begin{align}
(\alpha-i)_+-(\alpha-(i-1))_+=a_i',\label{f1}\\
(\alpha'-i)_+-(\alpha'-(i-1))_+=a_i''.\label{f2}
\end{align}
Therefore, we may choose the number $\delta$ sufficiently
small so that $\|a_i-a_i''\|<\epsilon$ for all $i\leq n$. We now rename the sequence $(a_i'')_{i=1}^n$ as $(a_i')_{i=1}^n$
and set $a_{n+1}'=e'$. From $\beta+f\sim \alpha'+e'$ we get that $\sum_{i=1}^{n+1} b_i'\sim \sum_{i=1}^{n+1} a_i'$.
This completes the induction. 

Continuing the induction up to $n=m$ we find $(a_i')_{i=1}^m$
and $(b_i')_{i=1}^m$ that satisfy \eqref{kk1}-\eqref{kk4}.

For the last part of the proof we split the analysis in to cases, $m$ even and $m$ odd. 

Suppose that $m=2k+1$. We have
\begin{align}\label{mo}
\frac{\sum_{i=1}^{2k+1} a_i'}{2k+1}\sim \frac{\sum_{i=1}^{2k+1} b_i'}{2k+1}.
\end{align}
Let $a'$ denote the sum on the left side of the last equation, and $b'$  the sum on the right. 
Let us show that $\|a'-a\|<2r+2\epsilon$ and $\|b-b'\|<2r+2\epsilon$.
Since $a_i'\lhd a_{i+1}'$ for all $i$ and $\|a_i'\|\le 1$ for all $i$, we have $a_i'\le a_{i+1}'$ for all $i$. Hence, 
\begin{align*}
\frac{2\sum_{i=1}^ka_{2i-1}'+a_{2k+1}'}{2k+1}\le\frac{\sum_{i=1}^{2k+1} a_i'}{2k+1}\le \frac{a_1'+2\sum_{i=1}^ka_{2i+1}'}{2k+1}.
\end{align*}
Using that $\|a_i'-a_i\|<\epsilon$ for $i$ odd in the above inequalities we obtain
\begin{align*}
\frac{2\sum_{i=1}^ka_{2i-1}+a_{2k+1}}{2k+1}-\epsilon \le \frac{\sum_{i=1}^{2k+1} a_i'}{2k+1}\le \frac{a_1+2\sum_{i=1}^ka_{2i+1}}{2k+1}+\epsilon.
\end{align*}
It follows now from the inequalities
\begin{align*}
\frac{2\sum\limits_{i=1}^k\xi_{2i-1}(t)+\xi_{2k+1}(t)}{2k+1}\le t+2r+\epsilon,\quad
t-2r-\epsilon\le \frac{\xi_1(t)+2\sum\limits_{i=1}^k\xi_{2i+1}(t)}{2k+1}, 
\end{align*}
that 
\begin{align*}
a-2r-2\epsilon\le\frac{\sum_{i=1}^{2k+1} a_i'}{2k+1}\le a+2r+2\epsilon.
\end{align*}
Therefore $\|a-a'\|<2r+2\epsilon$.

Let us show that $\|b-b'\|<2r+2\epsilon$. Using that  $b_i'\leq b_{i+1}'$ for $i=1, 2, \ldots, 2k$, that $b_{2k+1}'\le b_{2k+2}$, and that $\|b_i'-b_i\|<\epsilon$ for all $i$ even, we obtain the inequalities
\begin{align*}
\frac{2\sum_{i=1}^kb_{2i}}{2k+1}-\epsilon\le \frac{\sum_{i=1}^{2k+1}b_i'}{2k+1}\le \frac{2\sum_{i=1}^{k}b_{2i}+b_{2k+2}}{2k+1}+\epsilon.
\end{align*}
It follows from the estimates
\begin{align*}
\frac{2\sum_{1}^k\xi_{2i}(t)}{2k+1}\ge t-2r-\epsilon,\quad \frac{\xi_0(t)+2\sum_{1}^k\xi_{2i}(t)}{2k+1}\le t+2r+\epsilon,
\end{align*}
that
\begin{align*}
b-2r-2\epsilon\le \frac{\sum_{i=1}^{2k+1}b_i'}{2k+1}\le b+2r+2\epsilon.
\end{align*}
Hence $\|b-b'\|<2r+2\epsilon$.

We have found $a',b'\in A^+$ such that $a'\sim b'$, $\|a'-a\|<2r+2\epsilon$ and $\|b-b'\|<2r+2\epsilon$.
Therfore $d_U(a,b)\leq 4r+4\epsilon$. Since $\epsilon>0$ is arbitrary the desired result follows.

For the case that $m=2k$ we take $a'=\frac{1}{2k}\sum_{i=1}^{2k} a_{i}'$ and $b'=\frac{1}{2k}\sum_{i=1}^{2k} b_{i}'$,
and we reason similarly to how we did in the odd case to obtain that 
$\|a'-a\|<2r+2\epsilon$ and $\|b-b'\|<2r+2\epsilon$. 
\end{proof}

\begin{corollary}\label{completeness} Let $A$ be a C*-algebra with the property (I) of Theorem \ref{1}. The following propositions hold true:

(i) If $a$ and $b$ are positive elements of $A$ such that $d_W(a,b)<r$, then for all $\epsilon>0$ there exists $b'\in A^+$ such that $\|a-b'\|<4r$ and $d_U(b,b')<\epsilon$.

(ii) The set of positive elements of $A$ is complete with respect to the pseudometric $d_U$.

\end{corollary}
\begin{proof}
(i) We may assume without loss of generality that $a$ and $b$ are contractions.
We may also assume that $A$ is $\sigma$-unital by passing to
the subalgebra $\mathrm{Her}(a,b)$ if necessary (the property (I) holds for hereditary
subalgebras by Proposition \ref{classI}). Let $c\in A^+$ be strictly positive. By the
continuity of the pseudometrics $d_U$ and $d_W$ (see Lemma \ref{continuous}), it is enough
to prove the desired proposition assuming that $a$ and $b$ belong to a dense subset of $A^+$.
Thus, we may assume that $a,b\in \mathrm{Her}((c-\delta)_+)$ for some $\delta>0$.
From $d_W(a,b)<r$ and the proof of Theorem \ref{1} we get that there is 
$x\in \mathrm{Her}((c-\delta)_+)$ such that 
\[
\|a-x^*x\|<2r,\quad\|b-xx^*\|<2r.
\] 
Let $e\in A^+$ be a positive contraction that is a unit for the subalgebra $\mathrm{Her}((c-\delta)_+)$. 
Then $x^*x+e\sim_{ap} xx^*+e$. This implies that for all $\epsilon>0$ there is a unitary $u$ in $(A\otimes \mathcal{K})^\sim$ such that 
\[ 
\|u^*eu-e\|<\epsilon,\quad \|u^*x^*xu-xx^*\|<\epsilon.
\]
Set $eubu^*e=b'$. If we take $\epsilon$ small enough such that 
\[
\|a-x^*x\|<2r-\epsilon,\quad\|b-xx^*\|<2r-\epsilon,
\] 
we then have the following estimates:
\begin{align*}
& \|a-b'\|\le \|a-ubu^*\|<4r-2\epsilon+\|uxx^*u^*-x^*x\|<4r,\\
& \|u^*b'u-b\|\le \|u^*eubu^*eu-ebu^*eu\|+\|bu^*eu-be\|<2\epsilon.
\end{align*}
From here part (i) of the corollary follows.

(ii) Let $(c_i)_{i=1}^\infty$ be a sequence of positive elements of $A$ that is Cauchy with respect to the pseudometric $d_U$. In order to show that $(c_i)_{i=1}^\infty$ converges it is enough to show that it has a convergent
subsequence. We may assume, by passing to a subsequence if necessary, that $d_U(c_i,c_{i+1})<\frac{1}{2^i}$ for all $i\geq 1$. Using mathematical induction we will construct a new sequence $(c_i')_{i=1}^\infty$ such that 
\begin{align}\label{sequence}
\| c'_i-c'_{i+1}\|<\frac{1}{2^{i-3}}, \quad d_U(c_i, c_i')<\frac{1}{2^i},
\end{align}
for all $i$.

For $n=1$ we set $c_1=c_1'$. Suppose that we have constructed $c_i'$, for $i=1, 2,\ldots, n$, and let us construct $c_{n+1}'$. We have $d_U(c_{n+1},c_{n})<\frac{1}{2^n}$ and $d_U(c_{n}',c_{n})<\frac{1}{2^n}$
(by the induction hypothesis). Hence $d_U(c_{n}',c_{n+1})<\frac{1}{2^{n-1}}$, and so $d_W(c_n', c_{n+1})<\frac{1}{2^{n-1}}$ (by Lemma \eqref{continuous}). 
Applying part (i) of the corollary to $a=c_n'$ and $b=c_{n+1}$, we find a positive element $d$ such that 
\[
\|c_{n}'-d\|<\frac{1}{2^{n-3}},\quad d_U(c_{n+1}, d)<\frac{1}{2^{n+1}}.
\]
Setting $c_{n+1}'=d$ completes the induction.

By \eqref{sequence} the sequence $(c'_i)_{i=1}^\infty$ is a Cauchy sequence with respect to the norm of $A$. Hence, it converges to an element $c\in A^+$. Also by \eqref{sequence} we have that $d_U(c_i, c_i')<\frac{1}{2^i}$ for all $i$. Hence, $d_U(c_i,c)\le d_U(c_i, c_i')+d_U(c_i',c)\to 0$. That is, $(c_i)_{i=1}^\infty$ converges to $c$ in the
pseudometric $d_U$. Thus, $A^+$ is complete with respect to $d_U$.
\end{proof}

\subsection{Approximate existence theorem.}
Let $A$ be a C*-algebra and $h_A$ a strictly positive element of $A$. The main result of this subsection, Theorem \ref{existence} below, states 
that every morphism $\alpha\colon \Cu(C_0(0,1])\to \Cu(A)$ in the category $\CCu$ such that $\alpha([\mathrm{id}])\leq [h_A]$, may be 
approximated in the pseudometric $d_W$ by a morphism of the form $\Cu(\phi)$, with $\phi\colon C_0(0,1]\to A$ a C*-algebra homomorphism.

\begin{lemma}\label{cuntzorder}
Let $A$ be a C*-algebra. The following propositions hold true:

(i) If $a$ and $b$ are two positive elements of $A$ such that $a\cuntzle_{Cu} b$, then for every $\epsilon>0$ there is $b'\in M_2(A)^+$ 
such that $b'\sim_{Cu} b$ and
\[
\left\|
\begin{pmatrix}
a & 0\\
0 & 0
\end{pmatrix}-b'
\right\|<\epsilon.
\]

(ii) If $a$ and $b$ are two positive elements of $A\otimes \mathcal K$ such that $a\cuntzle_{Cu} b$ then for every $\epsilon>0$ there exists 
$b'\in (A\otimes \mathcal K)^+$ such that $b'\sim_{Cu} b$ and $\|a-b'\|<\epsilon$.
\end{lemma}
\begin{proof}
(i) Let $\epsilon>0$ be given. Since $a\cuntzle_{Cu} b$, by \cite[Lemma 2.2]{kirchberg-rordam} there exists $d\in A$ 
such that $(a-\epsilon/2)_+=d^*bd$. Consider the vector $c=(b^{\frac{1}{2}}d, \delta b^{\frac{1}{2}})$, where $\delta>0$. Then
\begin{align*}
cc^*=b^{\frac{1}{2}}dd^*b^{\frac{1}{2}}+\delta^2 b\quad\text{and}\quad
c^*c=
\begin{pmatrix}
(a-\epsilon/2)_+ & \delta d^*b\\
\delta bd & \delta^2 b
\end{pmatrix}.
\end{align*}
We may choose $\delta$  small enough such that 
\[
\left\|
\begin{pmatrix}
a & 0\\
0 & 0
\end{pmatrix}-c^*c
\right\|<\epsilon.
\]
Since $\delta^2b\leq cc^*\leq (\delta^2+\|d\|^2)b$, we have $cc^*\sim_{Cu} b$. Thus, the desired result follows letting $b'=c^*c$.

(ii) We may assume without loss of generality that $A$ is stable. This implies that for every $b\in (A\otimes \mathcal K)^+$ 
there is $b'\in A^+$ that is Murray-von Neumann equivalent to $b$, where $A$ is being identified with the top corner of 
$A\otimes \mathcal K$. Thus, we may assume without loss of generality that $b\in A^+$. Every positive element $a$ in 
$(A\otimes \mathcal K)^+$ is approximated by the elements $p_nap_n\in M_n(A)$ (with $p_n$ the unit of $M_n(A^\sim)$). Therefore, 
we may also assume without loss of generality that $a\in M_n(A)$ for some $n$. So we have $a,b\in M_n(A)^+$ for some $n$. 
Now the existence of  $b'\in M_{2n}(A)^+$ with the desired properties is guaranteed by part (i) of the lemma.
\end{proof}

\begin{lemma}\label{interpolation}
Let $A$ be a C*-algebra and let $(x_k)_{k=0}^{n}$ be elements of $\Cu(A)$ such that $x_{k+1}\ll x_k$ for all $k$. There is 
$a\in (A\otimes \mathcal K)^+$, with $\|a\|\leq 1$, such that $[a]=x_0$ and $x_{k+1}\ll[(a-k/n)_+]\ll x_k$ for $k=1,\ldots, n-1$. 
\end{lemma}
\begin{proof}
Let $\epsilon>0$. Let $a'_{n}\in (A\otimes \mathcal K)^+$ be such that $[a'_n]=x_n$ and $\|a'_n\|\leq \epsilon$.
Repeatedly applying Lemma \ref{cuntzorder} (ii),  we can find positive elements $(a'_k)_{i=0}^{n-1}$ such that $[a'_k]=x_k$ 
and $\|a'_k-a'_{k+1}\|<\epsilon$ for $k=1,\ldots, n-1$. For all $k$ we have $\|a'_0-a'_k\|<k\epsilon$. It follows from Lemma \ref{continuous} that $d_W(a_0',a_k')<k\epsilon$. Hence, 
\[
(a'_k-2k\epsilon)_+\cuntzle_{Cu} (a'_0-k\epsilon)_+\cuntzle_{Cu} a'_k.
\]
Since $x_{k+1}\ll x_k$ for all $k$, we can 
choose $\epsilon$ small enough such that
\begin{align*}
x_0=[a'_0]\geq x_1 \geq [(a'_0-\epsilon)_+] \geq x_2 &\geq [(a'_0-2\epsilon)_+] \geq \ldots \\
\ldots &\geq [(a'_0-(n-1)\epsilon)_+]\geq x_n.
\end{align*}
Set $a'_0/(n\epsilon)=a$. Then $[(a'_0-k\epsilon)_+]=[(a-k/n)_+]$ for all $k$. The lemma now follows by noticing that 
$\|a'_n\|\leq \epsilon$ and $\|a_0-a'_n\|<(n-1)\epsilon$ imply that $\|a\|\leq 1$.
\end{proof}

\begin{theorem}\label{existence}
Let $A$ be a C*-algebra and let $h_A$ be a strictly positive element of $A$. Let $\alpha\colon \Cu(C_0(0,1])\to \Cu(A)$ be 
a morphism in $\CCu$ such that $\alpha([\mathrm{id}])\le [h_A]$. Then for every $\epsilon>0$ there exists $\phi\colon C_0(0,1]\to A$ 
such that $d_W(\Cu(\phi),\alpha)<\epsilon$.
\end{theorem}
\begin{proof}
Let $\epsilon>0$ be given and let $n$ be such that $1/2^{n-1}<\epsilon$. Set $\alpha([e_t])=x_t$ for $t\in [0,1]$. By Lemma 
\ref{interpolation}, we can find $a\in (A\otimes \mathcal K)^+$ such that $\|a\|\leq 1$, $[a]=x_0$, and 
\begin{align}\label{int}
x_{(k+1)/2^n}\ll [(a-k/2^n)_+]\ll x_{k/2^n}
\end{align}
for $k=1,\ldots, 2^n-1$. Let $\delta>0$ be such that \eqref{int} still holds after replacing $a$ by $(a-\delta)_+$. This is 
possible since
\[
[(a-k/2^n)_+]=\sup_{\delta>0} [(a-\delta-k/2^n)_+].
\]
We have $[a]=\alpha([\mathrm{id}])\le [h_A]$. By \cite[Lemma 2.2]{kirchberg-rordam}, there exists $d\in A\otimes\mathcal K$ such that $(a-\delta)_+=dh_Ad^*$. 
Set $h_A^{1/2}d^*dh_A^{1/2}=a'$. Then  $a'$ is in $A^+$ and is Murray-von Neumann equivalent to $(a-\delta)_+$. It follows that $(a'-t)_+$ 
is Murray-von Neumann equivalent to $(a-\delta-t)_+$ for all $t\in [0,1]$. Therefore, $[(a'-k/2^n)_+]=[(a-\delta-k/2^n)_+]$ 
for $k=1,\ldots, 2^n-1$. So we have found a positive element $a'$ in $A^+$ such that
\begin{align*}
x_{(k+1)/2^n}\ll [(a'-k/2^n)_+]\ll x_{k/2^n}
\end{align*}
for $k=1,\ldots, 2^n-1$. Notice also that $\|a'\|=\|(a-\delta)_+\|<1$.

Let $\phi\colon C_0(0,1]\to A$ be such that $\phi(\mathrm{id})=a'$. Then 
\begin{align*}
\Cu(\phi)([e_{k/2^n}])\le \alpha([e_{k/2^n}])\,\text{ and }\,\alpha([e_{(k+1)/2^n}])\le \Cu(\phi)([e_{k/2^n}]).
\end{align*}
Any interval of length $1/2^{n-1}$ contains an interval of the form $(k/2^n, (k+1)/2^n)$ for some $k$. Thus, for 
every $t\in [0,1]$ there exists $k$ such that $(k/2^n, (k+1)/2^n)\subseteq (t, t+1/2^{n-1})$. It follows that
\begin{align*}
\Cu(\phi)([e_{t+1/2^{n-1}}])\le \Cu(\phi)([e_{k/2^n}])\le \alpha([e_{k/2^n}])\le \alpha([e_t])
\end{align*}
and
\begin{align*}
\alpha([e_{t+1/2^{n-1}}])\le \alpha([e_{(k+1)/2^n}])\le \Cu(\phi)([e_{k/2^n}])\le \Cu(\phi)([e_t]).
\end{align*}
These inequalities imply that  $d_W(\Cu(\phi),\alpha)\le 1/2^{n-1}<\epsilon$.
\end{proof}

\subsection{Weak cancellation in $\Cu(A)$}
\begin{proposition}\label{cancellation}
 Suppose that $(A\otimes \mathcal K)^\sim$ has the property (I) of Theorem \ref{1}. Then 
$Cu(A)$ has weak cancellation.
\end{proposition}
\begin{proof}
 Suppose that $[a]+[c]\wayb [b]+[c]$ for $[a]$, $[b]$, and $[c]$ in $\Cu(A)$. Let us choose
$a$, $b$, and $c$, such that $ac=bc=0$. Taking supremum over $\delta>0$ in $[(b-\delta)_+]+[(c-\delta)_+]$
we get that $[a]+[c]\le [(b-\delta)_+]+[(c-\delta)_+]$ for some $\delta>0$. 
Hence, for every $\epsilon>0$ there are $a_1$ and $c_1$ in $(A\otimes \mathcal K)^+$ such that
\begin{align*}
a_1+c_1 \in \mathrm{Her}&((b-\delta)_+ +(c-\delta)_+),\\
a_1\sim (a-\epsilon)_+,& \quad c_1\sim (c-\epsilon)_+, \hbox{ and }a_1c_1=0.
\end{align*}
We assume that $\epsilon<\delta/2$. Let us show that $a_1$ is Cuntz smaller than $b$. 

Let $g\in C_0(0,1]$ be such that $0\le g(t)\le 1$, $g(t)=1$ for $t\geq \delta-\epsilon$ and $g(t)=0$ for $t\leq \delta/2$.  
Then $g((c-\epsilon)_+)+g(b)$ is a unit for $a_1$ and $c_1$. 

We have $g(c_1)\sim g((c-\epsilon)_+)$. Let  $x$ be such that $g(c_1)=xx^*$ and $g((c-\epsilon)_+)=x^*x$.
From  $(g(b)+x^*x)xx^*=xx^*$ we deduce that  $(1-(g(b)+x^*x))x=0$. Let $w\in (A\otimes \mathcal K)^\sim$
be given by
\[
w=x+\sqrt{1-(g(b)+x^*x)}.
\]
Then we have $w^*w=1-g(b)$. From $a_1g(c_1)=0$ and $g(c_1)=xx^*$ we get that  $a_1x=0$.
Also $a_1(1-(g(b)+x^*x))=0$. We conclude that $aw=0$.
Let $\tilde b$ be defined by $ww^*=1-\tilde b$. Since we have assumed that the property (I) holds in $(A\otimes \mathcal K)^\sim$,
we have $w^*w+1\sim_{ap} ww^*+1$. From this we deduce $1-w^*w\sim_{ap} 1-ww^*$, i.e., $g(b)\sim_{ap} \tilde b$.
So $\tilde b\sim_{Cu} g(b)\cuntzle_{Cu} b$.  On the other hand, from 
$a_1w=0$ we deduce that $a_1\tilde b=a_1$, and so $a_1\leq \|a_1\|\tilde b$. Hence $a_1\cuntzle_{Cu}\tilde b\cuntzle b$.

We have $[(a-\epsilon)_+]=[a_1]\leq [b]$ for all $\epsilon>0$. Letting $\epsilon\to 0$ we get $[a]\leq [b]$
as desired.
\end{proof}
\subsection{Proof of Theorem \ref{2}}
\begin{proof}[Proof of Theorem \ref{2}]
The uniqueness of the homomorphism $\phi$ is clear by Theorem \ref{1}. Let us prove its existence.
By Theorem \ref{existence}, for every $n$ there exists $\phi_n\colon C_0(0,1]\to A$ such that 
$d_W(\mathrm{Cu}(\phi_n),\alpha)<1/2^{n+2}$. It follows from Theorem \ref{1} that
\begin{align*}
d_U(\phi_n(\mathrm{id}), \phi_{n+1}(\mathrm{id}))\le 4d_W(\mathrm{Cu}(\phi_n), \mathrm{Cu}(\phi_{n+1}))<1/2^n.
\end{align*}
This implies that $(\phi_n(\mathrm{id}))_n$ is a Cauchy sequence with respect to the pseudometric $d_U$. By Corollary \ref{completeness},  $A^+$ is complete with respect to $d_U$. Hence, there exists $\phi\colon C_0(0,1]\to A$ such that 
$d_U(\phi(\mathrm{id}), \phi_n(\mathrm{id}))\to 0$.
We have,
\begin{align*}
d_W(\mathrm{Cu}(\phi),\alpha) &\le d_W(\mathrm{Cu}(\phi), \mathrm{Cu}(\phi_n))+d_W(\mathrm{Cu}(\phi_n), \alpha),\\
& \le d_U(\mathrm{Cu}(\phi), \mathrm{Cu}(\phi_n))+d_W(\mathrm{Cu}(\phi_n), \alpha)\to 0
\end{align*}
So $d_W(\mathrm{Cu}(\phi),\alpha)=0$. By Propositions and Proposition \ref{metric}, $d_W$ is a metric. Therefore $\mathrm{Cu}(\phi)=\alpha$.
\end{proof}

\section{Examples and counterexamples}
\subsection{Algebras with the property (I)} 
The following proposition provides us with examples of C*-algebras with the property (I) of Theorem \ref{1}.
\begin{proposition} \label{classI}
The following propositions hold true.

(i) If $A$ is a C*-algebra of stable rank 1 then (I) holds in $A$.

(ii) If $X$ is a locally compact Hausdorff space such that  $\dim X\leq 2$ and $\check H^2(X)=0$, then 
(I) holds in $(C_0(X)\otimes \mathcal K)^\sim$.

(iii) If (I) holds in $A$ it also holds in  every hereditary subalgebra and every quotient of 
$A$.

(vi) If $A\cong \varinjlim A_i$ and (I) holds in the C*-algebras $A_i$ then
it  also holds in $A$.
\end{proposition}

\begin{proof}
(i) Let $x,e\in A$ be as in Theorem \ref{1} (I). Let $B$ be the smallest hereditary subalgebra 
of $A$ containing $x^*x$ and $xx^*$. Then $B$ has stable rank 1, and $e$ is a unit for $B$. It is well known that in a 
C*-algebra of stable rank 1 Murray-von Neumann equivalent positive elements are approximately
unitarily equivalent in the unitization of the algebra. Therefore, there are unitaries $u_n\in B^\sim$, $n=1,2,\dots$,
such that $u_n^*x^*xu_n\to xx^*$. We also have $u_n^*eu_n=e$ for all $n$, since $e$ is a unit for $B$.
Hence $u_n^*(x^*x+e)u_n\to xx^*+e$, as desired.

(ii) Let $x,e\in (C_0(X)\otimes \mathcal K)^\sim$ be as in Theorem \ref{1} (I). For every $t\in X$
the operators $x^*(t)x(t)+e(t)$ and $x(t)x^*(t)+e(t)$, in $\mathcal K^\sim$, are approximately 
unitarily equivalent, since $\mathcal K^\sim$ has stable rank 1. Let us denote by $\lambda\in \R$ the scalar 
such that $x^*x+e-\lambda \cdot 1\in C_0(X)\otimes \mathcal K$ and $xx^*+e-\lambda \cdot 1\in C_0(X)\otimes \mathcal K$. Then the selfadjoint
elements $x^*x+e-\lambda \cdot 1$ and $xx^*+e-\lambda \cdot 1$ have the same eigenvalues for any
point $t\in X$, and so by Thomsen's \cite[Theorem 1.2]{thomsen} they are approximately unitarily equivalent in 
$C_0(X)\otimes \mathcal K$. (Thomsen's result is stated for selfadjoint elements of 
$C_0(X)\otimes M_n$, but it easily extends to selfadjoint elements of $C_0(X)\otimes \mathcal K$). 
It follows that $x^*x+e$ and $xx^*+e$ are approximately unitarily equivalent in 
$(C_0(X)\otimes \mathcal K)^\sim$. 

(iii) The property (I) passes to hereditary subalgebras because approximate Murray-von Neumann
equivalence does (by Corollary \ref{mvnher}). 

In order to consider quotients by closed two-sided ideals we first make the following claim:
for every $\epsilon>0$ there is $\delta>0$ such that if $\|x(1-e)\|<\delta$ and $\|(1-e)x\|<\delta$,
with $e$ a positive contraction, then $d_W(x^*x+e,xx^*+e)<\epsilon$. In order to prove this we 
notice that the inequality $d_W(x^*x+e,xx^*+e)<\epsilon$ is implied by a finite set of relations of Cuntz comparison
on positive elements obtained by functional calculus on $x^*x+e$ and $xx^*+e$ (see the proofs of 
Theorem \ref{existence} and Lemma \ref{WTlimit} (ii)). Using the continuity of the functional calculus,
the argument used  in the implication (II)$\Rightarrow$(I) of Theorem \ref{1} can still be carried out, approximately,
to obtain this finite set of Cuntz comparisons.

Let us suppose that the algebra $A$ has the property (I). Let $x,e\in A/I$ be elements in a quotient
of $A$ such that $ex=xe=x$, and $e$ is a positive contraction. Let $\tilde x$ and $\tilde e$ be lifts
of $x$ and $e$, with $\tilde e$ a positive contraction. Let $(i_\lambda)$ be an approximate identity of $I$.
Let $\tilde e_\lambda\in A$ be the positive contraction defined by 
$1-\tilde e_\lambda=(1-\tilde e)^{1/2}(1-i_\lambda)(1-\tilde e)^{1/2}$. Then $\tilde e_\lambda$ is a 
lift of $e$ for all $\lambda$, and $(1-\tilde e_\lambda)\tilde x, \tilde x(1-\tilde e_\lambda)\to 0$.
Thus, we can find lifts $\tilde x$ and $\tilde e_\lambda$ of $x$ and $e$, such that $\|(1-\tilde e_\lambda)\tilde x\|<\delta$ and 
$\|\tilde x(1-\tilde e_\lambda)\|<\delta$ for any given $\delta>0$. By the claim made in the previous paragraph we can choose
$\delta$ such that $d_W(\tilde x^*\tilde x+\tilde e,\tilde x^*\tilde x+\tilde e)<\epsilon$, for any given $\epsilon>0$. Since 
$A$ has the property (I), we have by Theorem \ref{1} that $d_U(\tilde x^*\tilde x+\tilde e,\tilde x^*\tilde x+\tilde e)<4\epsilon$.
Passing to the quotient by $I$ we get  $d_U(x^*x+e,x^*x+e)<4\epsilon$, and since $\epsilon$ is arbitrary we are done.

(iii) Let $x,e\in A$  be as in Theorem \ref{1} (I). We may approximate these elements by the images of 
elements $x',e'\in A_n$, with $e'$ a positive contraction, within an arbitrary degree of proximity. 
By possibly moving the elements $x'$ and $e'$ further along the inductive limit, we may assume that
$e'$ is approximately a unit for $x'$. We can then use the claim established in the proof of (ii), to get that $d_W((x')^*x'+e',x'(x')^*+e')$ can be made arbitrarily small (choosing $x'$ and $e'$ suitably). 
Since $A_n$ has the property (I), we have
that $d_U((x')^*x'+e',x'(x')^*+e')$ can be arbitrarily small. Going back to the limit algebra this implies that $d_U(x^*x+e,xx^*+e)$ is arbitrarily small, and so it is 0. 
\end{proof}

\begin{example}
Let $D$ denote the unit disc in $\R^2$ and $U$ its interior. Let $B\subseteq M_2(D)$ be 
the hereditary subalgebra 
\[\begin{pmatrix} C(D) & C_0(U)\\ C_0(U) &C_0(U)\end{pmatrix}.\]
By Propositions \ref{classI} (ii) and (iv), (I) holds in $B$. Thus, the Cuntz semigroup
functor classifies the homomorphisms from $C_0(0,1]$ to $B$ up to stable approximate unitary equivalence.
Let us show that, unlike the case of stable rank 1 algebras, stable approximate unitary equivalence
and approximate unitary equivalence do not agree in $B$.
Let $p\in B$ be the rank 1 projection $\begin{pmatrix} 1 & 0\\ 0 &0\end{pmatrix}$ and let $q\in B$ be a rank 1
projection that agrees with $p$ on the boundary of $D$, and such that the projection induced by $1-q$ in $D/\!\!\sim$, the disc
with the boundary points identified, is nontrivial. Then $p$ and $q$ are Murray-von Neumann equivalent projections,
and so they are stably unitary equivalent. However, if there were $u\in B^\sim$ unitary such that $u^*pu=q$, then the partial
isometry $v=u^*(1-p)$ would be constant on $\mathbb{T}$ and such that $v^*v=1-q$ and  $vv^*=1-p$ is trivial. This would
contradict the nontriviality of $1-q$ in $D/\!\!\sim$. 
\end{example}
 
Examples of C*-algebras that do not have the property (I) are not hard to come by. If a unital C*-algebra
$A$ has (I), then for any two projections $p$ and $q$ in $A$ such that $p\sim q$,
we have that $p+1\sim_{ap} q+1$ by (I). From this we deduce by functional calculus on $p+1$ and $q+1$
that $1-p\sim 1-q$. Thus, any unital C*-algebra where Murray-von Neumann equivalence of projections does
not imply that they are unitary equivalent does not have (I). In particular, the algebra $B^\sim$, with $B$ 
as in the previous example, does not have (I).  

\subsection{The isometry question}
The following question was posed to us  by Andrew Toms:
if $A$ has stable rank 1, is it true that $d_W=d_U$?
We formulate this question here for the algebras
covered by Theorem \ref{1}.

\emph{Question.} Suppose that $A$ has the property (I) of Theorem \ref{1}. Is it true that $d_W=d_U$?

We do not know the answer to this question, even in the case of stable rank 1 algebras.
Proposition \ref{isometry} below provides some evidence that the answer is yes.

\begin{lemma}\label{WTlimit} Let $A=\varinjlim (A_i, \phi_{i,j})$ be the C*-algebra inductive limit of the sequence of C*-algebras 
$(A_i)_{i=1}^\infty$ with connecting homomorphisms $\phi_{i,j}\colon A_i\to A_j$. Let $a,b\in A_k^+$ for some $k$. Then 

(i) $d_{U}^{A_i}(a_i,b_i)\to d_{U}^A(a_\infty,b_\infty)$ as $i\to \infty$, and

(ii) $d_{W}^{A_i}(a_i,b_i)\to d_{W}^A(a_\infty,b_\infty)$ as $i\to \infty$,

\noindent where $a_i$ and $b_i$ denote the images of $a$ and $b$ by the homomorphism $\phi_{k,i}$, 
for $i=k+1, k+2,\ldots, \infty$.
\end{lemma}

\begin{proof}

(i)  We clearly have $d_{U}^{A_n}(a_n,b_n)\geq d_U^{A_{n+1}}(a_{n+1},b_{n+1})\geq d_U^A(a_\infty,b_\infty)$ 
for all $n\geq 1$. Therefore, it is enough to show that for every $\epsilon>0$ 
there is $n$ such that $d_U^{A_n}(a_n,b_n)\leq
d_U^A(a_\infty,b_\infty)+\epsilon$.

Let us denote $d_U^A(a_\infty,b_\infty)$ by $r$ and let $\epsilon>0$. Let $u\in (A\otimes \mathcal K)^\sim$ be a unitary
such that $\|ua_\infty u^*-b_\infty\|<\epsilon+r$. Since $A\otimes \mathcal K=\varinjlim A_i\otimes \mathcal K$,
there are $n$ and a unitary $u'\in (A_n\otimes\mathcal K)^\sim$ such that 
$\|u'a_n(u')^*-b_n\|<\epsilon+r$. Hence $d_U^{A_n}(a_n,b_n)\leq d_U^A(a_\infty,b_\infty)+\epsilon$.

(ii) We may assume without loss of generality that $k=1$. As before, we have $d_{W}^{A_n}(a_n,b_n)\geq d_W^{A_{n+1}}(a_{n+1},b_{n+1})\geq d_W^A(a_\infty ,b_\infty )$ for all $n\geq 1$. Thus,  we need to show that for every $\epsilon>0$ there is $n$ such that $d_W^{A_n}(a_n,b_n)\leq
d_W^A(a_\infty ,b_\infty)+\epsilon$.

Let us denote  $d_W(a_\infty,b_\infty)$ by $r$ and let $\epsilon>0$. Let us choose a grid of points $\{t_i\}_{i=1}^m$ in $(0,1]$ such that 
$t_i<t_{i+1}$ and $|t_i-t_{i+1}|<\epsilon$ for $i=1,\ldots,m-1$ (e.g., choose $m\geq 1/\varepsilon$ and $t_i=i/m$ for $i=1,\ldots,m$).
From the Cuntz inequality $e_{t_i+r+\epsilon/4}(a_\infty)\cuntzle_{Cu} e_{t_i}(b_\infty)$ and \cite[Lemma 2.2]{kirchberg-rordam}, we deduce that there exists $d_i\in A$ 
such that $e_{t_i+r+\epsilon/2}(a_\infty)=d_ie_t(b_\infty)d_i^*$. Since $A$ is the inductive limit of the C*-algebras $A_n$, we can find $n$ and 
$d_i'\in A_n$ such that
\[
\|e_{t_i+r+\epsilon/2}(a_n)-d_i'e_{t_i}(b_n)(d_i')^*\|<\epsilon/2.
\] 
By \cite[Lemma 2.2]{kirchberg-rordam} applied in the algebra $A_n$, we have that $e_{t_i+r+\epsilon}(a_n)\cuntzle_{Cu} e_{t_i}(b_n)$ in $A_n$.
Let us choose a value of $n$ such that this inequality holds in $A_n$ for all $i=1,2,\dots,m-1$, and such that we also have 
$e_{t_i+r+\epsilon}(b_n)\cuntzle_{Cu} e_{t_i}(a_n)$ for all $i=1,2,\dots,m-1$. 

Let $t\in [0,1]$. Let $i$ be the smallest integer such that $t\leq t_i$. Then $[t_i,t_i+r+\epsilon]\subseteq [t,t+r+2\epsilon]$. 
We have the following inequalities in $A_n$:
\[
e_{t+r+2\epsilon}(a_n)\cuntzle_{Cu} e_{t_i+r+\epsilon}(a_n)\cuntzle_{Cu}
e_{t_i}(b_n)\cuntzle_{Cu} e_t(b_n).
\]
The same inequalities hold after interchanging $a_n$ and $b_n$. Thus, $d_W^{A_n}(a_n,b_n)\leq r+2\epsilon$. 
\end{proof}

\begin{proposition}\label{isometry}
 Let $A$ be such that $A\otimes \mathcal K$ is an inductive limit of algebras of the form
$C(X_i)\otimes K$, with $\dim X_i\leq 2$, $\check H^2(X_i)=0$. Then the pseudometrics $d_U$ and $d_W$ 
agree on the positive elements of $A$.
\end{proposition}
\begin{proof}
We may assume without loss of generality that $A$ is stable.
Let $A=\varinjlim (C_0(X_i)\otimes \mathcal K,\phi_{i,i+1})$.
Since both $d_U$ and $d_W$ are continuous 
(by Lemma \ref{continuous}), it is enough to show that they are equal on a dense subset of $A^+$. Thus, we may assume that 
$a$ and $b$ belong to the image in $A$ of some algebra $C_0(X_i)\otimes \mathcal K$. Furthermore, in order to show that $d_U(a,b)=d_W(a,b)$, 
it is enough to show, by Lemma \ref{WTlimit},  that this equality holds on all the algebras 
$C_0(X_j)\otimes \mathcal K$, with $j\geq i$. Thus, we may assume that the algebra $A$ is itself of the 
form $C_0(X)\otimes \mathcal K$, with $\dim X\le 2$ and $\check H^2(X)=0$. Finally, since 
$\bigcup_{n=1}^\infty M_n(C_0(X))$ is dense in $C_0(X)\otimes \mathcal K$, we may assume that $a,b\in M_n(C_0(X))$
for some $n\in \N$.

So let $a,b\in M_n(C_0(X))$ be positive elements. Set $d_W(a,b)=r$. Then for every $x\in X$ we have $d_W(a(x),b(x))\leq r$, 
where $d_W$ is now taken in the C*-algebra $M_n(\C)$. From the definition of $d_W$ we see that this means that for every $t>0$, 
the number of eigenvalues of $a(x)$ that are less than $t$ is less than the number of eigenvalues of $b(x)$ that are less than $t+r$, 
and vice-versa, the number of eigenvalues of  $b(x)$ less than $t$, is less than the number of eigenvalues of $a(x)$ less than $t+r$. 
By the Marriage Lemma, this means that the eigenvalues of $a(x)$ and $b(x)$ may be matched in such a way that the distance between 
the paired eigenvalues is always less than $r$. By \cite[Theorem 1.2]{thomsen}, this implies that $d_U(a,b)<r$.
\end{proof}

\subsection{Counterexamples.}
The counterexamples of this subsection are C*-algebras that not only do not have the property (I),
but moreover the Cuntz semigroup functor does not distinguish the stable approximate unitary classes of homomorphisms 
from $C_0(0,1]$ to the algebra. 
\begin{example}\label{sphere}
Let $S^2$ denote the 2-dimensional sphere. Let us show that there are homomorphisms 
$\phi,\psi\colon C_0(0,1]\to M_2(C(S^2))$ 
such that $\Cu(\phi)=\Cu(\psi)$ but $\phi$ is not stably approximately unitarily equivalent to $\psi$.

Let $\lambda_1$ and $\lambda_2$ be  continuous functions from $S^2$ to $[0,1]$
such that $\lambda_1>\lambda_2$, $\min \lambda_2=0$, and $\min\lambda_1\leq \max \lambda_2$.
Let $P$ and $E$ be rank one projections in $M_2(C(S^2))$ such that $E$ is trivial and $P$ is non-trivial. 
Consider the positive elements
\[
a=\lambda_1 P+\lambda_2(1_2-P)\,\hbox{ and }\, b=\lambda_1 E+\lambda_2(1_2-E),
\] 
where $1_2$ denotes the unit of $M_2(C(S^2))$. Let us show that for every non-zero function 
$f\in C_0(0,1]$ we have $f(a)\sim f(b)$. In view of the computation of the Cuntz semigroup of $S^2$ 
obtained in \cite{robert}, it is enough to show that the rank functions of $f(a)$ and $f(b)$ are equal 
and non-constant. We have $f(a)=f(\lambda_1)P+f(\lambda_2)(1-P)$ and $f(b)=f(\lambda_1)E+f(\lambda_2)(1-E)$.
It is easily verified  that the rank functions of $f(a)$ and $f(b)$ are both equal to $\mathds{1}_U+\mathds{1}_V$, 
where $U=\{x\mid f(\lambda_1(x))\neq 0\}$, $V=\{x\mid f(\lambda_2(x))\neq 0\}$, and $\mathds{1}_U$ and $\mathds{1}_V$ 
denote the characteristic functions of $U$ and $V$. Since $\min \lambda_2=0$, the open set $V$ is a proper subset of $S^2$. So if $V$ is non-empty, then the function $\mathds{1}_U+\mathds{1}_V$ is non-constant. On the other hand, if $V$ is empty, then $f$ is 0 on the interval $[0,\max\lambda_2]$; in particular, $f(\min\lambda_1)=0$. Thus, $U$ is a proper subset of $S^2$ in this case, and so $\mathds{1}_U+\mathds{1}_V$ is again non-constant.

Let $\phi,\psi\colon C_0(0,1]\to M_2(C(S^2))$ be the homomorphisms such that
$\phi(\mathrm{id})=a$ and $\psi(\mathrm{id})=b$. It follows from the discussion in the previous paragraph 
that $\Cu(\phi)=\Cu(\psi)$. Let us show that $\phi$ and $\psi$ are not stably approximately
unitarily equivalent. 

Let $t=\max(\frac{\lambda_1}{\lambda_2})$ and $r>0$. Then
\[
e_t\left(\frac{a}{\lambda_1}\right)=(1-t)P\, \hbox{ and }\, 
e_{t+r}\left(\frac{b}{\lambda_1}\right)=(1-t-r)E.
\]
In order that $e_{t+r}(b/\lambda_1)$ be Cuntz smaller than $e_t(a/\lambda_1)$ the value
of $r$ must be at least $1-t$. Thus, $d_W(\frac{a}{\lambda_1},\frac{b}{\lambda_1})\geq 1-t$.
Hence $\frac{a}{\lambda_1}\nsim_{ap}\frac{b}{\lambda_1}$, and so $a\nsim_{ap}b$.
It follows that  $\phi$ and $\psi$ are not stably approximately unitarily equivalent.
\end{example}

Next we construct a simple AH C*-algebra for which the Cuntz semigroup functor does not
classify the homomorphisms from $C_0(0,1]$ into the algebra.

Let us recall the definition given in \cite{villadsen} of a diagonal homomorphism from $C(X)\otimes \mathcal K$ to $C(Y)\otimes \mathcal K$ (here $X$ and $Y$ are compact Hausdorff spaces). 
Let $(p_i)_{i=1}^n$ be mutually orthogonal projections in $C(Y)\otimes\mathcal  K$ and let $\lambda_i\colon Y\to X$, $i=1,2,\dots,n$, be continuous maps. Let us define a homomorphism $\phi\colon C(X)\to C(Y)\otimes\mathcal  K$ by
\[\phi(f)=\sum_{i=1}^n (f\circ\lambda_i)p_i.\] 
The homomorphism $\phi$ gives rise to a homomorphism $\tilde\phi$
from $C(X)\otimes\mathcal K$ to $C(Y)\otimes\mathcal K$ as follows: $\tilde\phi$ is the composition of 
$\phi\otimes \id\colon C(X)\otimes \mathcal K\to C(Y)\otimes\mathcal  K\otimes\mathcal K$ with 
$\id\otimes \alpha\colon C(Y)\otimes\mathcal K\otimes \mathcal K\to C(Y)\otimes\mathcal K$, where $\alpha$ is  some isomorphism map from 
$\mathcal K\otimes \mathcal K$ to $\mathcal K$. A homomorphism $\tilde\phi$ obtained
in this way is said to be a diagonal homomorphism arising from the data $(p_i,\lambda_i)_{i=1}^n$
(the choice of $\alpha$ does not change the approximate unitary equivalence class of $\tilde\phi$).

\begin{theorem}\label{example}
There exists a simple stable AH C*-algebra $A$, and homomorphisms $\phi,\psi\colon C_0(0,1]\to A$, such that $\Cu(\phi)=\Cu(\psi)$
but $\phi$ and $\psi$ are not approximately unitarily equivalent.
\end{theorem}
\begin{proof}
Let us define the sequence of topological spaces $(X_i)_{i=1}^\infty$ by $X_1=\mathrm{CP}(1)$ and $X_{i+1}=X_i\times \mathrm{CP}(n_i)$, 
where $n_i=2\cdot(i+1)!$ and $\mathrm{CP}(n)$ denotes the complex projective space of dimension $2n$.
For every $n$ let us denote by $\eta_{n}$ the rank one projection in $C(\mathrm{CP}(n))\otimes \mathcal K$ 
associated to the canonical line bundle of $\mathrm{CP}(n)$. For every $i$ let $\pi_i\colon X_{i+1}\to X_i$ denote the projection map onto $X_i$.
Let $\tilde\phi_i\colon C(X_i)\otimes K\to C(X_{i+1}\otimes \mathcal K)$ denote the diagonal homomorphism
given by the data $(1,\pi_i)\cup (\eta_{n_i}^j,\delta_{y_i^j})_{j=1}^i$, where $(\eta_{n_i}^j)_{j=1}^i$ are mutually orthogonal projections
all Murray-von Neumann equivalent to $\eta_{n_i}$, and $\delta_{y_i^j}\colon X_{i+1}\to X_i$ is the constant map equal to $y_i^j\in X_i$ 
for $j=1,2,\dots,i$. It is possible, and well known, to choose the points $y_i^j$ in such a way that the inductive limit 
$A=\varinjlim (C(X_i)\otimes K,\phi_i)$ is a simple C*-algebra (see \cite{villadsen}). Let us show that this inductive limit $A$ provides 
us with the desired example. 

Let $a,b\in C(X_1)\otimes \mathcal K$ be the two positive elements constructed in the proof of Theorem \ref{sphere} (notice that $X_1$ 
is homeomorphic to $S^2$). Set $\phi_{1,i}(a)=a_i$ and $\phi_{1,i}(b)=b_i$ for $i=2,3,\dots,\infty$. For $i=2,\dots,\infty$, let us denote 
by $\phi_{a_i}$ and $\psi_{b_i}$ the homomorphisms from $C_0(0,1]$ to $A_i$ associated to the positive elements $a_i$ and $b_i$. Since
$\Cu(\phi_{a_i})=\Cu(\psi_{b_i})$, we have $\Cu(\phi_{a_\infty})=\Cu(\psi_{b_\infty})$. Let us show, on the other hand, that the homomorphisms 
$\phi_{a_\infty}$ and $\psi_{b_\infty}$ are not approximately unitarily equivalent. Equivalently, let us show that $d_U(a_\infty,b_\infty)>0$. 
By Lemma \ref{WTlimit}, it suffices to show that $d_U(a_n,b_n)$ does not tend to 0. Let us show that $d_U(a_n,b_n)\geq (\min \lambda_1)(1-\max(\lambda_2/\lambda_1))$ for all $n$, where $\lambda_1$ and $\lambda_2$ are the functions used in the definition of $a$ and $b$ in Theorem \ref{sphere}.

Let us denote by $\tilde\eta_i\in C(X_i)\otimes \mathcal K$ the projection $e_0\otimes 1\otimes\dots \otimes \eta_i\otimes \dots\otimes 1$,
where  $\eta_i$ is placed in the $i$-th position of the tensor product. Here we view $ C(X_i)\otimes \mathcal K$ as the tensor product 
\[
(C(\mathrm{CP}(1))\otimes \mathcal K)\otimes C(\mathrm{CP}(n_2))\otimes\dots \otimes C(\mathrm{CP}(n_i)). 
\]
Let $p$ be an arbitrary  projection in 
$C(X_1)\otimes \mathcal K$. It was observed in \cite{villadsen} that the image of $p$ by $\phi_{1,i}$ is Murray-von Neumann equivalent 
to the projection
\begin{align*}
(p\otimes 1\otimes \dots \otimes 1)\oplus k_1\tilde\eta_1 \oplus k_2\tilde\eta_2\oplus \dots\oplus k_i\tilde\eta_i,
\end{align*}
where $k_i\in \N$. In this expression the multiplication by the coefficients $k_i$ indicates the orthogonal sum of $k_i$ copies of the
projection $\tilde\eta_i$. In a similar manner, one can show that for every scalar function $\lambda\in C(X_1)$ 
the image of $\lambda p$ by $\phi_{1,i}$ is Murray-von Neumann equivalent to
\begin{align*}
\lambda (p\otimes 1\otimes \dots \otimes 1)\oplus \bigoplus_j \lambda(y_1^j)\tilde\eta_1 \oplus 
\bigoplus_j \lambda(y_2^j)\tilde\eta_2\oplus\dots\oplus \bigoplus_j \lambda(y_i^j)\tilde\eta_i.
\end{align*}
Since $a$ and $b$ have both the form $\lambda_1 p\oplus \lambda_2 q$, for some projections $p$
and $q$, and scalar functions $\lambda_1$ and $\lambda_2$, the formula above allows us to compute 
the images of $a$ and $b$ in $C(X_i)\otimes \mathcal K$, i.e., the elements $a_i$ and $b_i$, up 
to Murray-von Neumann equivalence. Thus, $a_i$ is Murray-von Neumann equivalent to
\begin{align*}
&\lambda_1 \tilde\eta_1\oplus \bigoplus_j \lambda_1(y_1^j)\tilde\eta_1 \oplus 
\bigoplus_j \lambda_1(y_2^j)\tilde\eta_2\oplus\dots\oplus \bigoplus_j \lambda_1(y_i^j)\tilde\eta_i\oplus\\
&\lambda_2 \tilde\eta_1'\oplus \bigoplus_j \lambda_2(y_1^j)\tilde\eta_1 \oplus 
\bigoplus_j \lambda_2(y_2^j)\tilde\eta_2\oplus\dots\oplus \bigoplus_j \lambda_2(y_i^j)\tilde\eta_i,
\end{align*}
where $\tilde\eta_1'=(1_2-\eta_1)\otimes 1\otimes\dots\otimes 1$. A similar expression holds for $b_i$.

Let $a'_i=a_i/\lambda_1$ and $b_i'=b_i/\lambda_1$. Let $t=\max(\lambda_2/\lambda_1)$. 
Let us show that $d_W(a_i',b_i')\geq 1-t$. We have that $(a_i'-t)_+$ is Murray-von Neumann
equivalent to
\begin{align}\label{bigoplus}
(1-t)\tilde\eta_1\oplus&\bigoplus_j \alpha_{1,j}(y)\tilde\eta_1 \oplus 
\bigoplus_j \alpha_{2,j}(y)\tilde\eta_2\oplus\dots \oplus\bigoplus_j \alpha_{i,j}(y)\tilde\eta_i\oplus\\
&\bigoplus_j \beta_{1,j}(y)\tilde\eta_1 \oplus\bigoplus_j \beta_{2,j}(y)\tilde\eta_2\oplus\dots
\oplus\bigoplus_j \beta_{i,j}(y)\tilde\eta_i,\nonumber
\end{align}
where $\alpha_{k,j}(y)=(\frac{\lambda_1(y_k^j)}{\lambda_1(y)}-t)_+$ and $\beta_{k,j}(y)=(\frac{\lambda_2(y_k^j)}{\lambda_1(y)}-t)_+$
for $k,j=1,\dots,i$. It follows that 
\[
[(a_i'-t)_+]\leq [\tilde\eta_1]+\sum_{j=2}^i 2k_j[\tilde\eta_j]
\]
in the Cuntz semigroup of $C(X_i)\otimes \mathcal K$. For the element $(b_i'-t)_+$ an expression 
identical to \eqref{bigoplus} may be found, except that the first summand of \eqref{bigoplus} is replaced with the 
term $(1-t)(1\otimes \dots\otimes 1)$. It follows that for all $r<1-t$ we have $[1\otimes\dots \otimes 1]\leq [(b_i'-t-r)_+]$. 
Since we do not have $[1\otimes\dots \otimes 1]\leq [\tilde\eta_1]+\sum_{j=2}^i 2k_j[\tilde\eta_j]$ (because the total Chern class of
the projection on the right side is nonzero), we conclude that $d_W(a_i',b_i')\geq 1-t$. By Lemma \ref{continuous} we have $d_U(a_i',b_i')\geq 1-t$. Hence 
\[
d_U(a_i,b_i)\geq (\min \lambda_1)\cdot d_U(a_i',b_i')\geq (\min \lambda_1)\cdot (1-\max(\lambda_2/\lambda_1)).\qedhere
\] 
\end{proof}

\section{Classification by the functor $\Cu(\cdot\otimes \mathrm{Id})$}
Let $A$ and $B$ be C*-algebras. For $a\in (A\otimes \mathcal K)^+$ a contraction, let us denote by $d_W^a$
the pseudometric on the Cuntz category morphisms from $\Cu(A)$ to $\Cu(B)$ given by
\[
d_W^a(\alpha,\beta):= d_W(\alpha\circ\Cu(\phi_a),\beta\circ\Cu(\phi_a)),
\] 
where $\phi_a\colon C_0(0,1]\to A\otimes \mathcal K$ is such that $\phi(\mathrm{id})=a$. We consider
the set $\mathrm{Mor}(\Cu(A),\Cu(B))$ endowed with the uniform structure induced by all the pseudometrics $d_W^a$.
A basis of entourages for this uniform structure is given by the sets
\[
U_{F,\epsilon}=\{(\alpha,\beta)\mid d_W^a(\alpha,\beta)<\epsilon,a\in F\},
\]
where $\epsilon>0$ and $F$ runs through the finite subsets of positive contractions of $A\otimes \mathcal K$.

We will prove the following theorem, of which Theorem \ref{extension} 
of the introduction is an obvious corollary.

\begin{theorem}\label{extension2}
For every $\epsilon>0$ there is a finite set $F\subset C_0(0,1]\otimes C_0(0,1]$, and $\delta>0$,
such that
\[
(\Cu(\phi\otimes \mathrm{Id}),\Cu(\psi\otimes \mathrm{Id}))\in U_{F,\delta}\Rightarrow 
d_U(\phi(\id),\psi(\id))<\epsilon,
\]
for any pair of homomorphisms $\phi,\psi\colon C_0(0,1]\to A$, where the C*-algebra $A$ is an inductive limit of the form $\varinjlim C(X_i)\otimes\mathcal K$, with $X_i$ compact metric spaces and $\dim X_i\leq 2$ for all $i=1,2\dots$.
\end{theorem}

Before proving Theorem \ref{extension2} we need some preliminary definitions and results.
We will consider the relation of Murray-von Neumann equivalence on projections
in matrix algebras over possibly non-compact spaces. If $P$ and $Q$ are 
projections in the algebra $M_n(C_b(X))$ of continuous, bounded, matrix valued functions
on $X$, we say that $P$ and $Q$ are Murray-von Neumann
equivalent, and denote this by $P\sim Q$, if there is $v\in M_n(C_b(X))$
such that $P=vv^*$ and $Q=vv^*$. For a subset $U$ of $X$, assumed either open or closed, 
we say that $P$ is Murray-von Neumann equivalent to $Q$ on the set $U$ if the restrictions of
$P$ and $Q$ to $U$ are Murray-von Neumann equivalent in the algebra $M_n(C_b(U))$.

\begin{lemma}\label{pointsCW}
 Let $X$ be a finite CW-complex of dimension at most 2, and let $C$ be a closed subset of $X$. 
If $P$ and $Q$ are projections in $M_n(C(X))$ such that $P$ is Murray-von Neumann equivalent to $Q$
on the set $C$, then there exists a finite subset $F$ of $X\backslash C$ such that
$P$ is Murray-von Neumann equivalent to $Q$ on $X\backslash F$.
\end{lemma}

\begin{proof}
Let $X_1$ denote the 1-skeleton of $X$ and $(\Delta_i)_{i=1}^m$ the 2-cells of $X$. Suppose that 
$(\Delta_i)_{i=1}^{m_0}$ are the 2-cells intersected by the open set $X\backslash C$. Choose 
points $x_i\in \mathring{\Delta}_i\backslash C$ for $i\leq m_0$, and let $F$ be the set of these points. Since $X\backslash F$
contracts to $X_1\cup \bigcup_{i>m_0} \Delta_i$, it is enough to show that
$P$ is Murray-von Neumann equivalent to $Q$ on $X_1\cup \bigcup_{i>m_0} \Delta_i$ (see \cite[Theorem 1]{vaserstein}). 
Let $v$ be a partial isometry defined on $\bigcup_{i>m_0} \Delta_i$ such that $P=vv^*$ and $Q=v^*v$
on $\bigcup_{i>m_0} \Delta_i$ ($v$ exists by hypothesis). Let us show that $v$ extends 
to $X_1\cup \bigcup_{i>m_0}\Delta_i$. For this,
it is enough to show that $v$ extends from  $X_1\cap \bigcup_{i>m_0} \Delta_i$ to $X_1$. This
is true by \cite[Proposition 4.2 (1)]{phillips} (applied to 1-dimensional spaces).
\end{proof}

\begin{proposition}\label{CWcompare}
Let $X$ be a finite CW-complex of dimension at most 2. Let $\epsilon>0$. 
Suppose that $a,b\in M_n(C(X))^+$ are of the form 
\begin{align}\label{aandb}
a=\sum_{j=1}^n P_j\lambda_j, \quad b=\sum_{j=1}^n Q_j\lambda_j,
\end{align}
where $(P_j)_{j=1}^n$ and $(Q_j)_{j=1}^n$ are sequences of orthogonal projections of rank 1, 
$(\lambda_j)_{j=1}^n$ is a sequence of scalar functions such that $\lambda_j\geq \lambda_{j+1}$ for $j=1,2\dots n-1$,
and 
\begin{align}
\sum_{j=1}^i P_j\sim\sum_{j=1}^i Q_j \hbox{ on the set }\{x\in X\mid\lambda_i(x)-\lambda_{i+1}(x)\geq \epsilon\},
\label{lammu2}
\end{align}
for $i=1,\dots,n$ (for $i=n$ we take $\lambda_{i+1}=0$ in \eqref{lammu2}). Then $d_U(a,b)<2\epsilon$.
\end{proposition}

\begin{proof}
Let $\epsilon>0$ and $a$ and $b$ be as in the statement of the lemma.
Let us perturb the elements $a$ and $b$ by modifying 
the functions $(\lambda_i)_{i=1}^n$ in the following way: For $i=1,2,\dots,n$, let us denote by
$C_i$ the set $\{x\in X\mid \lambda_i(x)-\lambda_{i+1}(x)\geq \epsilon\}$.
By \eqref{lammu2} and Lemma \ref{pointsCW}, there are finite sets $F_i\subseteq X\backslash C_i$ such that $\sum_{j=1}^i P_j$ is
Murray-von Neumann to $\sum_{j=1}^i Q_j$ on $X\backslash F_i$ for $i=1,2,\dots,n$. Let us choose
the sets $F_i$ so that they are disjoint for different $i$s (it is clear from the proof of Lemma \ref{pointsCW} that this
is possible). Furthermore, for every $x\in \bigcup_{i=1}^n F_i$ 
let us choose an open  neighbourhood $U(x)$ of $x$ such that $U(x)\cap U(x')=\varnothing$ for $x\neq x'$
and $U(x)\cap C_i=\varnothing$ for $x\in F_i$. Starting with $i=1$, and proceeding
to $i=2,\dots,n$, let us
perturb  the function $\lambda_{i+1}$ on the set $\bigcup_{x\in F_i} U(x)$ by an amount less than
$\epsilon$, and  so that $\lambda_{i+1}(x)=\lambda_i(x)$ for $x$ in some open set $V_i$ such that
$F_i\subset V_i$ and $\overline {V_i}\subseteq\bigcup_{x\in F_i} U(x)$.

Since the sets $\bigcup_{x\in F_i} U(x)$
are disjoint for different values of $i$, the resulting perturbations of $a$ and $b$ are within
a distance of $\epsilon$ of their original values. These perturbations,  which we continue to denote by $a$ and 
$b$,  satisfy that
\begin{align}\label{abconds}
&a =\sum_{j=1}^n  P_j\lambda_j ,\quad
b =\sum_{j=1}^n Q_j \lambda_j,\\
\label{abconds2}
&\sum_{j=1}^i P_j \sim \sum_{j=1}^i Q_j \hbox{ on }X\backslash V_i, \hbox{ for }i=1,2,\dots,n,\\
&V_i\subseteq \{x\mid \lambda_i(x)=\lambda_{i+1}(x)\}, \hbox{ and }V_i \hbox{ is open.}\label{abconds3}
\end{align}
The proposition will be proved once we show that, under the conditions \eqref{abconds}-\eqref{abconds3},
the elements $a$ and $b$ are Murray-von Neumann equivalent.
This amounts to finding a sequence of orthogonal projections $(R_i)_{i=1}^n$ in $M_n(C(X))$ such that
$a=\sum_{j=1}^n  R_j\lambda_j$, and $R_i\sim Q_i$ for $i=1,\dots,n$. Let us show that this is possible.

The sequence $(R_i)_{i=1}^n$ will be obtained by a series of modifications on the sequence
$(P_i)_{i=1}^n$. Let $k_0$ be the smallest index such that $P_{k_0}\nsim Q_{k_0}$. From 
$\sum_{j=1}^{k_0-1}P_i\sim \sum_{j=1}^{k_0-1}Q_i$ and \eqref{abconds2}, we get that 
$P_{k_0}\sim Q_{k_0}$ on $X\backslash V_{k_0}$ (since there is cancellation of projections over
spaces of dimension at most 2). Let $v$ be the partial isometry defined on $X\backslash V_{k_0}$ such that 
$P_{k_0}=vv^*$ and $Q_{k_0}=v^*v$ on $X\backslash V_{k_0}$. It is guaranteed by \cite[Proposition 4.2 (1)]{phillips} 
that $v$ can be extended to a partial isometry $w$ on $X$ such that $w^*w=Q_{k_0}$ and $ww^*\leq P_{k_0}+P_{k_0+1}$. Set $ww^*=P_{k_0}'$, with $w$ being such an extension of $v$. Then
$P_{k_0}'$ is such that $P_{k_0}'\sim Q_{k_0}$, $P_{k_0}'\leq P_{k_0}+P_{k_0+1}$, and 
$P_{k_0}'(x)=P_{k_0}(x)$ for all $x\in X\backslash V_{k_0}$.
Let  $P_{k_0+1}'$ be the projection such that 
$P_{k_0}'+P_{k_0+1}'=P_{k_0}+P_{k_0+1}$.
We have
\[
P_{k_0}\lambda_{k_0}+P_{k_0+1}\lambda_{k_0+1}=P_{k_0}'\lambda_{k_0}+P_{k_0+1}'\lambda_{k_0+1}.
\]
Thus, replacing $P_{k_0}$ and $P_{k_0+1}$ by $P_{k_0}'$ and $P_{k_0+1}'$ respectively, we obtain
a new sequence of projections $(P_i)_{i=1}^n$ that satisfies \eqref{abconds} and \eqref{abconds2},
and also $P_{k}\sim Q_{k}$ for $k\leq k_0$. Continuing this process we obtain the desired sequence $(R_i)_{i=1}^n$.
\end{proof}

\begin{proof}[Proof of Theorem \ref{extension2}]
Let $\epsilon>0$ (and assume $\epsilon<1$). Let $g_\epsilon\in C_0(0,1]$ be a function such that 
$g_\epsilon(t)=\frac{\epsilon}{t}$ for $t\in [\epsilon,1]$, and 
$0\le g_\epsilon(t)\le 1$ for $t\in (0,1]$. Let $F\subseteq C_0(0,1]\otimes C_0(0,1]$
be the set $F=\{\mathrm{id}\otimes \id,\id\otimes g_\epsilon\}$. Let us prove
that
\[
(\Cu(\phi\otimes \mathrm{Id}),\Cu(\psi\otimes \mathrm{Id}))\in U_{F,\frac{\epsilon^2}{2}}\Rightarrow 
d_U(\phi(\id),\psi(\id))<2\epsilon+\frac{\epsilon^2}{2},
\]
where $\phi$, $\psi$, and $A$ are as in the statement of the theorem. 
Let us express what we wish to prove in terms of positive contractions (via the
bijection $\phi\mapsto \phi(\id)$). For $a,b\in A$ positive contractions, we have 
\begin{align*}
d_W^{\id\otimes\id}(a\otimes \id,b\otimes\id)=d_W(a\otimes\id,b\otimes\id),\\
d_W^{\id\otimes g_\epsilon}(a\otimes\id,b\otimes\id)=d_W(a\otimes g_\epsilon,b\otimes g_\epsilon). 
\end{align*}
Thus, we want to show that 
\begin{align}\label{withcontractions}
\begin{array}{l}
d_W(a\otimes\id,b\otimes\id)<\frac{\epsilon^2}{2},\\
d_W(a\otimes g_\epsilon,b\otimes g_\epsilon)<\frac{\epsilon^2}{2},
\end{array}
\Rightarrow d_U(a,b)<2\epsilon+\frac{\epsilon^2}{2},
\end{align}
for $a$ and $b$ positive contractions.

Let us first show that if we have \eqref{withcontractions} for the C*-algebras $(A_i)_{i=1}^\infty$
of a sequential inductive system, then we also have \eqref{withcontractions} for their inductive limit $A$.
By the continuity of the pseudometrics $d_W$ and $d_U$
(see  Lemma \ref{continuous}), it is enough to prove \eqref{withcontractions} assuming
that $a$ and $b$ belong to a dense subset of the positive contractions of $A$. Thus, we may assume
that $a$ and $b$ are the images in $A$ of positive contractions in some C*-algebra $A_i$, $i\in \N$.
Suppose we have $a',b'\in A_i$ such that their images in $A$ satisfy the inequalities of the left side
of \eqref{withcontractions}. By Lemma \ref{WTlimit} (ii), it is possible to  move $a'$ and $b'$ along the inductive limit
to a C*-algebra $A_j$, $j\geq i$, so that these same inequalities hold in the C*-algebra $A_j$.
We conclude that $d_U^{A_j}(\phi_{i,j}(a'),\phi_{i,j}(b'))<\epsilon$. Moving $a$ and $b$ back to
the limit we get the right side of \eqref{withcontractions}.

From the discussion of the previous paragraph, it is enough to prove \eqref{withcontractions} for $A=C(X)\otimes \mathcal K$, with $X$ a compact metric space of dimension at most 2. Moreover, since
a compact metric space of dimension at most 2 is a sequential projective limit
of finite CW-complexes of dimension at most 2 (see \cite[Theorem 1.13.5]{engelking}), we
are reduced to proving \eqref{withcontractions} for the case that $A=C(X)\otimes \mathcal K$, where $X$ is a finite
CW-complex of dimension at most 2.


Let us suppose $A=C(X)\otimes \mathcal K$, where $X$ is a finite
CW-complex of dimension at most 2. It is enough to prove \eqref{withcontractions} assuming that
$a,b\in M_n(C(X))$ for some $n\in \N$. Moreover, by
Choi and Elliott's \cite[Theorem 1]{choi-elliott}, we may assume that $a(x)$ and $b(x)$ have distinct
eigenvalues (as matrices in $M_n(\C)$) for all $x\in X$. (Choi and Elliott's Theorem implies
that such a set is dense in the set of positive contractions of $M_n(C(X))$ for $\dim X\leq 2$.)
This implies (see the proof of \cite[Theorem 1.2]{thomsen}) that $a$ and $b$ have the form
\begin{align}\label{abPQ}
 a=\sum_{j=1}^n P_j\lambda_i \hbox{\, and \,}b=\sum_{j=1}^n Q_j\mu_i,
\end{align}
for some sequences of orthogonal projections of rank 1 $(P_i)_{i=1}^n$ and  $(Q_i)_{i=1}^n$, and scalar eigenfunctions 
$(\lambda_i)_{i=1}^n$ and  $(\mu_i)_{i=1}^n$, such that $1\geq \lambda_1(x)>\lambda_2(x)>\dots>0$ and $1\geq \mu_1(x)>\mu_2(x)>\dots>0$.

From $d_W(a\otimes\id,b\otimes\id)<\frac{\epsilon^2}{2}$ we deduce that 
$d_W(a,b)<\frac{\epsilon^2}{2}$ 
(evaluating $\id$ at $t=1$), and so $\|\lambda_i-\mu_i\|<\frac{\epsilon^2}{2}$ for all $i$ (see the proof of Theorem \ref{isometry}). Let $b'\in M_n(C(X))$ be given by $b'=\sum_{i=1}^n Q_i\lambda_i$.
Then $d_U(b,b')<\epsilon^2/2$ and
\[
d_W(a\otimes g_\epsilon,b'\otimes g_\epsilon)\leq d_W(a\otimes g_\epsilon,b\otimes g_\epsilon)+d_W(b\otimes g_\epsilon,b'\otimes g_\epsilon)<\epsilon^2.
\]
The implication \eqref{withcontractions} will be proven once we have shown that 
\[
d_W(a\otimes g_\epsilon,b'\otimes g_\epsilon)<\epsilon^2\Rightarrow d_U(a,b')<2\epsilon.
\]
In order to prove this, it is enough to show that the left side of this implication
implies \eqref{lammu2} of Proposition \ref{CWcompare} (applied to the elements $a$ and $b'$). 
Let us choose $\epsilon'>0$  such that 
$d_W(a\otimes g_\epsilon,b'\otimes g_\epsilon)<\epsilon^2-\epsilon'\epsilon$. By the definition
of $d_W$ we have that
\[
(a\otimes g_\epsilon-(\epsilon-\epsilon'\epsilon))_+\cuntzle_{Cu} (b'\otimes g_\epsilon-(\epsilon-\epsilon^2))_+.
\]
Let us identify $M_n(C(X))\otimes C_0(0,1]$ with $M_n(C_0(X\times (0,1]))$ and 
express the Cuntz comparison above  in terms of the projections $(P_i)_{i=1}^n$ and $(Q_i)_{i=1}^n$, and the eigenfunctions $(\lambda_i)_{i=1}^n$. We get 
\begin{equation}\label{cuntzineq}
\sum_{j=1}^n P_j(x)(\lambda_j(x)g_\epsilon(t)-\epsilon+\epsilon'\epsilon)_+\cuntzle_{Cu}
\sum_{j=1}^n Q_j(x)(\lambda_j(x)g_\epsilon(t)-\epsilon+\epsilon^2)_+,
\end{equation} 
for $(x,t)\in X\times (0,1]$. Note: this Cuntz relation comparison is not to be understood as a pointwise relation, but
rather as a relation in the C*-algebra $M_n(C(X\times (0,1]))$.

For $i\leq n$ let us define the set 
\[
T_i=\{x\in X\mid \lambda_{i+1}(x)/\lambda_{i}(x)\leq 1-\epsilon\hbox{ and }\lambda_i(x)\geq \epsilon\}.
\]
Let $C_i\subseteq X\times (0,1]$ be the closed set $C_i=\{(x,\lambda_i(x))\mid x\in T_i\}$.  Restricting the 
Cuntz comparison \eqref{cuntzineq} to the set $C_i$, and using the definition 
of $g_\epsilon$, we get that
\begin{align*}
P_1 &\left(\frac{\lambda_1}{\lambda_i} - (1-\epsilon')\right)_+   + P_2 \left(\frac{\lambda_2}{\lambda_i}-(1-\epsilon')\right)_+ +\dots +\epsilon' P_i
\preceq_{Cu} \\
& Q_1 \left(\frac{\lambda_1}{\lambda_i}-(1-\epsilon)\right)_+ + Q_2 \left(\frac{\lambda_2}{\lambda_i}-(1-\epsilon)\right)_+ +\dots +
\epsilon Q_i,
\end{align*}
on the closed set $T_i$. It follows that $\sum_{j=1}^i P_j\cuntzle_{Cu} \sum_{j=1}^i Q_j$ on $T_i$. 
In the same way we can prove that $\sum_{j=1}^i Q_j\cuntzle_{Cu} \sum_{j=1}^i P_j$ on $T_i$, and so
$\sum_{j=1}^i P_j\sim \sum_{j=1}^i Q_j$ on $T_i$. If  
$\lambda_i(x)-\lambda_{i+1}(x)\geq \epsilon$ then $\lambda_{i+1}(x)/\lambda_{i}(x)\leq 1-\epsilon$
and $\lambda_i(x)\geq \epsilon$. Hence, $\{x\in X\mid \lambda_i(x)-\lambda_{i+1}(x)\geq \epsilon\}\subseteq T_i$. 
Therefore, the elements $a$ and $b'$ satisfy the condition \eqref{lammu2} of Proposition \ref{CWcompare}. 
This completes the proof of the theorem.
\end{proof}

\end{document}